\documentclass{amsart}

\usepackage{amssymb}
\usepackage{geometry}
\usepackage{color}
\usepackage{graphicx}
\usepackage{bm}
\usepackage{enumitem}
\usepackage{comment}
\usepackage{dsfont}

\newtheorem{thm}{Theorem}[section]
\newtheorem*{thm*}{Theorem}
\newtheorem{cor}[thm]{Corollary}
\newtheorem{lem}[thm]{Lemma}
\newtheorem{prop}[thm]{Proposition}
\theoremstyle{definition}
\newtheorem{defn}[thm]{Definition}
\newtheorem{rem}[thm]{Remark}

\usepackage[colorlinks=true,urlcolor=blue, citecolor=red,linkcolor=blue,linktocpage,pdfpagelabels, bookmarksnumbered,bookmarksopen]{hyperref}

\numberwithin{equation}{section}

\newcommand{\RR}{\mathbb{R}}

\newcommand{\TT}{\bm{T}}
\newcommand{\hx}{\bm{x}_{\perp}}
\newcommand{\xx}{\bm{x}}
\newcommand{\yy}{\bm{y}}
\newcommand{\XX}{\bm{\chi}}
\newcommand{\ea}{\bm{e}}
\newcommand{\hy}{\bm{y}_{\perp}}

\newcommand{\vo}{\bm{\omega}}
\newcommand{\vv}{\bm{v}}

\title{Evolution of viscous vortex filaments and soliton-type propagation}

\author{Marco A. Fontelos}
\thanks{M. A. F. has been supported by the MICINN (Spain) under the Research Grant PID2023-150166NB-I00.}
\address{Marco A. Fontelos\\Instituto de Ciencias Matemáticas (ICMAT), CSIC-UAM-UCM-UC3M, Madrid, Spain}
\email{marco.fontelos@icmat.es}

\author{Mikel Ispizua}
\thanks{M. I. has been supported by the MICINN (Spain) under the Research Grant PID2024-156169NB-I00.}
\address{Mikel Ispizua\\Department of Mathematics, EHU, Bilbao, Spain}
\email{mikel.ispizua@ehu.eus}

\author{Luis Vega}
\thanks{L. V. has been supported by MICINN (Spain) projects Severo Ochoa CEX2021-001142 and PID2024-156169NB-I00, and by Eusko Jaurlaritza project IT1615-22 and the BERC program.}
\address{Luis Vega\\Department of Mathematics, EHU, Bilbao, Spain\\Basque Center for Applied Mathematics (BCAM), Bilbao, Spain}
\email{lvega@bcamath.org}

\begin{document}

\begin{abstract}
We study the evolution of a viscous incompressible fluid whose initial vorticity is supported on a smooth open curve. Following the ideas in \cite{FontelosVega}, we show that, for $\nu t\ll 1$ and sufficiently small Reynolds number $\Gamma/\nu$, the vorticity is described at leading order by a Lamb--Oseen type vortex concentrated around a curve evolving according to the binormal flow predicted by the localized induction approximation. The solution is written as an explicit leading-order profile plus a lower-order perturbation, which is controlled in a Morrey $\mathcal{M}^{\infty}$ norm.

Then, we apply this construction to the Hasimoto soliton. In this case, the estimates are uniform with respect to the torsion parameter, allowing us to consider a large-torsion regime in which the soliton undergoes a macroscopic displacement. We show that the corresponding Navier--Stokes solution contains a localized portion of the kinetic-energy distribution, associated with the binormal velocity, which remains concentrated inside a moving physical region and undergoes an order-one displacement during an admissible time interval.
\end{abstract}


\maketitle

\section{Introduction}

\subsection{Previous work}

The behavior of fluids constitutes one of the oldest and most challenging problems in mathematical physics. Vorticity filaments—one-dimensional structures where vorticity is concentrated—emerge as fundamental entities to understand phenomena ranging from turbulence generation to energy transfer within a fluid. Their formation, structure, and evolution have been central topics since the early days of fluid dynamics, addressed by figures such as Helmholtz, Lord Kelvin, and Kirchhoff (see \cite{SaffmanBook} for an overview and \cite{Darrigol} for historical context).

In this work we study the three-dimensional flow of a viscous incompressible fluid when the initial vorticity is concentrated along a curve $\XX(s)$ in $\RR^3$ parametrized by arclength $s$. In this setting, the governing equation of motion is the Navier--Stokes equation:
\begin{equation}\tag{N-S}\label{vorticity_NS}
\begin{cases}
\partial_t \boldsymbol{\omega} + (\bm{v} \cdot \nabla)\boldsymbol{\omega}= (\boldsymbol{\omega} \cdot \nabla)\bm{v} + \nu \Delta \boldsymbol{\omega}\quad&\text{in }\mathbb{R}^3\times(0,\infty),\\
\nabla\cdot\vv=0, \quad \bm{\omega}(t=0)=\Gamma\delta_{\XX}\TT(s)\quad&\text{in }\mathbb{R}^3,
\end{cases}
\end{equation}
where $\vv$ and $\vo$ represent the velocity and the vorticity of the fluid respectively, $\nu>0$ the viscosity, $\Gamma$ the circulation (or vortex strength), $\TT(s)$ the tangent vector of $\XX(s)$ and $\delta_{\XX}$ the Dirac delta measure supported on $\XX$. The vorticity is defined by $\vo=\nabla\wedge\vv$ and the velocity field is recovered from $\vo$ through the Biot--Savart law
\begin{equation}\label{definicion biot savart}
\vv(\xx)=K*\bm{\omega}(\xx)=\frac{1}{4\pi}\int_{\RR^3}\bm{\omega}(\bm{y})\wedge\frac{\xx-\bm{y}}{|\xx-\bm{y}|^3}{\rm d}\bm{y}.
\end{equation}
In the inviscid case, when $\nu=0$, the filament evolves under the velocity field generated by \eqref{definicion biot savart}, which blows up at every point on the filament. By replacing the zero-thickness filament by a tube of radius $\epsilon$, the singularity is removed and the velocity is given, up to $O(1)$ terms, by
\begin{equation}\label{definicion binormal flow 0}
\vv=-\frac{\Gamma}{4\pi}\kappa(s)\log(\epsilon)\bm{b}(s,t),
\end{equation}
where $\bm{b}(s,t)$ and $\kappa(s,t)$ denote the binormal vector of the Frenet--Serret frame and the curvature of the filament at position $s$ and time $t$, respectively.
This desingularization of \eqref{definicion biot savart} is the so-called \emph{localized induction approximation} (LIA), introduced by Da Rios in 1906 in \cite{DaRios}, see also \cite{RiccaNature}. In the LIA, one keeps only the logarithmic contribution produced by the local part of the filament and neglects the $O(1)$ terms, which contain the nonlocal influence of the distant parts of the curve. In this approximation, the velocity of each point of the filament is therefore determined only by the local geometry, through the curvature and the binormal direction.

The study of the binormal flow as a geometric evolution law and, in particular, the existence of coherent structures under its influence, is a topic of mathematical interest by itself. A fundamental result in this context is due to Hasimoto \cite{Hasimoto72}, who introduced a transformation relating the binormal flow to the cubic nonlinear Schrödinger equation (NLS). This connection makes it possible to construct explicit shape-preserving structures, including the solitons found by Hasimoto \cite{Hasimoto72} as well as the helical and screw-wave solutions studied by Kida \cite{Kida1981,Kida1982}. See also \cite{BanicaVega} for a recent survey on NLS and the binormal flow. These explicit solutions will be relevant in the final part of the paper, where we consider a Hasimoto soliton as a reference curve around which the vorticity is concentrated.

Concerning the rigorous study of vortex filaments in Euler flows, the existence of coherent vortex structures has been established in several important cases, including finite-core vortex rings \cite{Fraenkel,BergerFraenkel,FriedmanTurkington, Norbury}, traveling helices \cite{DavilaDelPinoMussoWei2022}, and leapfrogging vortex rings \cite{DavilaDelPinoMussoWei2024}. More recently, Jerrard and Seis \cite{JS2012} derived the binormal flow as the evolution law for vortex filaments of thickness $\varepsilon$ for general initial data under some assumptions on the distribution of the vorticity inside the core of the vortex.

In the viscous setting, Giga and Miyakawa proved in \cite{GigaMiyakawa} the existence of solutions to \eqref{vorticity_NS} with measure-valued initial data in Morrey spaces, under a smallness assumption on the initial data, which in this context corresponds to $\frac{\Gamma}{\nu}$ being sufficiently small. This smallness condition was later removed by Gallay and \v{S}ver\'ak for axisymmetric vortex rings. In \cite{GallaySverak}, Gallay and \v{S}ver\'ak showed that the vorticity field $\omega$ becomes smooth for any positive time and describes a vortex ring of thickness $\sqrt{\nu t}$ translating along its symmetry axis with a precise velocity predicted by Kelvin and later justified by Saffman  (see \cite{SaffmanBook}). In the regime $\sqrt{\nu t}\ll R$, where $R$ denotes the radius of the ring, the leading-order behavior of this velocity is given by the binormal flow. Finally, Bedrossian, Germain and Harrop-Griffiths \cite{BedrossianGermain HG} prove global solutions for initial data of arbitrary circulation which are perturbations of an infinite straight line or a smooth closed curve.

The simplest example of a viscous vortex is the \emph{Lamb--Oseen vortex}, which arises as an exact solution of \eqref{vorticity_NS} with initial condition
\[
\vo(\xx,0)=\Gamma\,\delta(x_1,x_2)\,\ea_3
\]
in Cartesian coordinates $\xx=(x_1,x_2,x_3)\in\RR^2\times\RR$. In this setting, the Navier--Stokes system reduces to the two-dimensional heat equation in the $(x_1,x_2)$-plane
\[
\partial_t\vo=\nu\Delta_{(x_1,x_2)}\vo,
\qquad (x_1,x_2,t)\in\RR^2\times\RR_+.
\]
The solution is explicit and is given by
\begin{align*}
\vo(r,t)&=\frac{\Gamma}{4\pi\nu t}e^{-\frac{r^2}{4\nu t}}\,\ea_3,\\
\vv(r,t)&=\frac{\Gamma}{2\pi r}\left(1-e^{-\frac{r^2}{4\nu t}}\right)\ea_{\theta},
\end{align*}
where $(r,\theta)$ denote polar coordinates in the $(x_1,x_2)$-plane.

In \cite{FontelosVega}, the first and third authors give a precise description of the solution of \eqref{vorticity_NS} when the initial data is a smooth, non-self-intersecting closed curve $\XX(s,0)$. In that work, this description is obtained by writing the vorticity as an expansion in the viscous scale $\sqrt{\nu t}$. In the regime where $\sqrt{\nu t}$ is sufficiently small, the leading-order term is a Lamb--Oseen vortex concentrated around a curve $\XX(s,t)$ evolving according to the binormal flow
\begin{equation}\label{definicion binormal flow}
\XX_t(s, t)=-\frac{\Gamma}{4\pi}\log{(\sqrt{\nu t})}\XX_s(s, t)\wedge\XX_{ss}(s, t),
\end{equation}
where here the thickness of the tube in \eqref{definicion binormal flow 0} is substituted by the viscous scale $\sqrt{\nu t}$. The next-order vorticity term is still singular as $\nu t\to0$, but is of lower order and contains the first local correction due to the curvature of the filament. This reflects the fact that, at small scales, the filament is described to leading order by a straight Lamb--Oseen vortex with curvature corrections, while the curve itself evolves macroscopically according to the binormal flow.

In order to state the main result of \cite{FontelosVega}, we first introduce some notation. For $R>0$, we define the tubular neighborhood of radius $R$ around $\XX(s,t)$ as
\begin{equation*}
\mathcal{T}_R(t):=\left\{\xx\in\mathbb{R}^3:\exists s \text{ such that } |\xx-\XX(s,t)|<R\right\}.
\end{equation*}
We say that $\mathcal{T}_R$ is regular if for every $\xx\in\mathcal{T}_R(t)$, there exists a unique point $\XX(s_0,t)$ on the curve such that
\[
|\xx-\XX(s_0,t)|=\min_s|\xx-\XX(s,t)|.
\]
As a sufficient condition for the existence of a regular neighborhood around $\XX(s, t)$ we assume (see Section \ref{subsec: filament and tubular neighborhood} for further details) that there exists an $R$ uniform in $t\in[0, T]$ such that
\begin{equation}\label{tubular intro}
R:=\frac{1}{4}\min\left\{\frac{1}{\kappa_{\max}},\frac{1}{2}\inf_{\substack{t\in[0,T]\\ d_{\rm arc}(s,s')>\pi/\kappa_{\max}}}|\XX(s,t)-\XX(s',t)|\right\},
\end{equation}
where
\[
\kappa_{\max}:=\sup_{t\in[0,T]}\sup_s|\partial_s\TT(s,t)|
\]
and
\[
d_{\rm arc}(s,s') :=
\begin{cases}
|s-s'|, & \text{if the curve is open},\\[6pt]
\displaystyle \min\{\,|s-s'|,\; L-|s-s'|\,\}, & \text{if the curve is closed of total length } L.
\end{cases}
\]
Inside this regular tubular neighborhood, local cylindrical coordinates $(r,\theta,s)$ are well defined. Using the Frenet--Serret frame $\{\TT(s,t),\bm{n}(s,t),\bm{b}(s,t)\}$, any point $\xx\in\mathcal{T}_R(t)$ can be written as
\[
\xx=\XX(s,t)+r\cos\theta\,\bm{n}(s,t)+r\sin\theta\,\bm{b}(s,t).
\]
In the statement, the following self-similar radial variable is also used:
\begin{equation}\label{def: selfsimilar variable}
\rho=\frac{r}{\sqrt{\nu t}}.
\end{equation}
The result proven in \cite{FontelosVega} can be stated as follows:
\begin{thm}[Fontelos--Vega]\label{thm: fontelos vega}
Let $\XX\in C^3$ be a closed curve in $\mathbb R^3$, evolving under the binormal flow, with initial data $\XX(s,0)=\XX_0(s)$. Assume that the radius $R$ of the tubular neighborhood given by \eqref{tubular intro} is well defined. Assume also that $\Gamma/\nu$ and $\nu T$ are sufficiently small. Then there exists a solution $\bm{\omega}(\xx,t)$ to \eqref{vorticity_NS}, defined on $(0,T)$, such that, for $\xx\in\mathcal{T}_R(t)$, its expression in the local coordinates $(r,\theta,s)$ is given by
\begin{align}
\bm{\omega}(r,\theta,s,t)&=\frac{1}{\nu t}\frac{\Gamma}{4\pi}e^{-\rho^2/4}\,\TT(s,t)+\frac{1}{(\nu t)^{1/2}}\frac{\Gamma\kappa}{8\pi}\rho e^{-\rho^2/4}\cos\theta\,\TT(s,t)\notag\\
&\quad+\frac{1}{(\nu t)^{1/2}}\left(\Omega^{c,(2)}_1(\rho)\cos\theta+\Omega^{s,(2)}_1(\rho)\sin\theta\right)\TT(s,t)+\tilde{\bm{\omega}}(r,\theta,s,t),
\label{eq:vorticity-expansion}
\end{align}
where
\[
\left|\Omega^{s,c,(2)}_1(\rho)\right|\leq C\frac{\Gamma^2}{\nu}(\rho+\rho^2)e^{-\rho^2/4},
\]
and
\[
\|\tilde{\bm{\omega}}\|_{L^2}^2(t)+\nu\int_0^t\|\nabla\tilde{\bm{\omega}}\|_{L^2}^2(t')\,dt'\leq C\Gamma^2(\nu t)|\log(\nu t)|^2.
\]
Here $C$ is a suitable constant independent of the parameters.
\end{thm}

\subsection{Main results}
One of the main contributions with respect to Theorem \ref{thm: fontelos vega} is that, in this work, open filaments are included. More precisely, we consider filaments which, in addition to satisfying the tubular neighborhood condition \eqref{tubular intro}, satisfy a chord--arc condition uniformly on $(0,T)$. Namely, there exists $c_0>0$, independent of $t$, such that
\begin{equation}\label{uniform_chord_arc_condition_intro}
|\XX(s,t)-\XX(s',t)|\geq c_0 |s-s'|
\end{equation}
for every $t\in[0,T]$ and every $s,s'\in\RR$. See Section \ref{subsec: filament and tubular neighborhood} for details. We remark that, although throughout the paper we work with open filaments for simplicity, the results could be extended to the closed case without substantial changes.

Another important contribution of the present work concerns the norm in which the error is estimated. In \cite{FontelosVega}, although the vorticity expansion is written pointwise in local coordinates, the remainder is controlled through an $L^2$-based energy estimate. Here, instead, the vorticity error is controlled in the Morrey $\mathcal{M}^\infty$ norm, which, for functions, is equivalent to the usual $L^\infty$ norm, giving a pointwise control of the remainder. Furthermore, for infinite-length filaments the energy-type norm used in \cite{FontelosVega} is no longer natural, since it is not finite. Working in the Morrey spaces used in \cite{GigaMiyakawa} overcomes this difficulty, as these spaces are well adapted to one--dimensional vorticity distributions.

Before stating the result, let us introduce the Morrey norms used in the theorem and throughout the paper. These norms will be discussed in more detail in Section \ref{sec: existence result}, where we study the existence of the solutions under consideration. Let $\mu$ be a Radon measure on $\RR^3$ and $1\leq p\leq\infty$. Following \cite{GigaMiyakawa}, we define
\[
\|\mu\|_p:=\sup_{\xx\in\RR^3,\ r>0}r^{-\frac{3}{p'}}|\mu|(B(\xx,r)),
\]
with $p'=\frac{p}{p-1}$. Here $|\mu|(B(x,r))$ denotes the total variation of $\mu$ in the ball of radius $r$ centered at $x$. For a time-dependent family $\mu(t)$, the corresponding time-weighted norm is
\[
\|\mu\|_{T,p}:=\sup_{0<t<T}(\nu t)^{\frac12-\frac{3}{2p}}\|\mu(t)\|_p.
\]

Finally, for practical reasons, throughout this work we use the \emph{parallel frame} $\{\TT,\ea_1,\ea_2\}$ attached to $\XX$, instead of the usual Frenet--Serret frame. Thus, a point $\xx$ in the tubular neighborhood $\mathcal{T}_R(t)$ of $\XX$ is written as
\[
\xx=\XX(s,t)+r\cos\varphi\,\ea_1(s,t)+r\sin\varphi\,\ea_2(s,t).
\]
The definition of this frame and the convenience of this choice are discussed in Section \ref{subsec: Parallel frame and local coordinates}. Now, let us state the main result of the present paper.
\begin{thm}\label{Teorema vorticidad intro}
Let $\XX\in C^3$ be an open curve in $\mathbb R^3$, evolving under the binormal flow, with initial data $\XX(s,0)=\XX_0(s)$. Assume that the radius $R$ of the tubular neighborhood given by \eqref{tubular intro} is well defined and that there exists $c_0>0$ such that the chord-arc condition \eqref{uniform_chord_arc_condition_intro} holds uniformly in time on $[0,T]$. Assume also that $\Gamma/\nu$ and $\nu T$ are sufficiently small. Then there exists a solution $\bm{\omega}(\xx,t)$ to \eqref{vorticity_NS}, defined on $(0,T)$, such that, for any $\xx\in\mathcal{T}_{\frac{R}{2}}(t)$, its expression in the local coordinates $(r,\varphi,s)$ is given by
\[
\bm{\omega}(r,\varphi, s,t)=\frac{\Gamma}{4\pi}\frac{e^{-\frac{\rho^2}{4}}}{\nu t}\TT(s,t)+\tilde{\bm{\omega}}(r,\varphi, s,t).
\]
Moreover, for every $\frac32\leq p\leq\infty$,
\begin{equation}\label{norma morrey introduccion}
\sup_{0<t<T}(\nu t)^{\frac12-\frac{3}{2p}}\|\tilde{\bm{\omega}}(t)\|_p\leq C\mathcal{C}_F(\XX,\Gamma,\nu T),
\end{equation}
where $\mathcal{C}_F(\XX,\Gamma,\nu T)$ is an explicit quantity depending on the geometry of the filament and on the parameters of the problem. In addition, for every $\frac32\leq p<3$,
\[
\|\tilde{\bm{\omega}}(t)\|_p\longrightarrow0\qquad\text{as }t\to0.
\]
Furthermore, the velocity induced by the perturbation satisfies
\begin{equation}\label{norma velocidad morrey intro}
\sup_{0<t<T}\|K*\tilde{\bm{\omega}}(t)\|_\infty\leq C\mathcal{C}_F(\XX,\Gamma,\nu T).
\end{equation}
\end{thm}

The proof of this result follows the general scheme established in \cite{FontelosVega}. We write
\[
\bm{\omega}=\bm{\omega}_0+\tilde{\bm{\omega}},
\]
where
\[
\bm{\omega}_0(\xx,t)=\frac{\Gamma}{4\pi}\frac{e^{-\frac{r^2}{4\nu t}}}{\nu t}\eta_R(r)\TT(s,t)
\]
and $\eta_R(r)$ is a cutoff function supported in $\mathcal{T}_R$. Then we study the Biot--Savart law associated with the leading-order vorticity term $\vo_0$. We derive an expansion for the induced velocity field, showing that it consists of three contributions: the local Lamb--Oseen velocity, the binormal velocity of the reference curve, and a remainder term whose dependence on the geometry of the filament is made explicit. Namely,
\[
K*\bm{\omega}_0:=\vv_0=\bm{v}_{L\text{-}O}+\XX_t+\bm v^*,
\]
where
\[
\bm{v}_{L\text{-}O}=\frac{\Gamma}{2\pi}\frac{1}{\sqrt{\nu t}}\frac{1}{\rho}\left(1-e^{-\frac{\rho^2}{4}}\right)\ea_\varphi,
\]
and $\XX_t$ is the binormal flow velocity \eqref{definicion binormal flow}. Special attention is paid to the lowest regularity needed on $\XX$ for the expansion to hold. In the estimate of $\vv^*$, one is naturally led to a scale-invariant quantity measuring the variation of $\TT'=\partial_s\TT$ along the filament. This term can be controlled by a Dini-type condition, although in the present work we impose the stronger assumption $\XX\in C^{2,\alpha}$, which is sufficient for the general estimates. See Lemma \ref{lem:BS-in-s} and Remark \ref{rem:Dini-condition} for details.

Then, defining
\[
\bm{v}=\bm{v}_0+\tilde{\bm{v}},
\]
and inserting the decompositions
\[
\bm{\omega}=\bm{\omega}_0+\tilde{\bm{\omega}},
\qquad
\bm{v}=\bm{v}_0+\tilde{\bm{v}},
\]
into \eqref{vorticity_NS}, we obtain an equation for the perturbation $\tilde{\bm{\omega}}$:
\begin{equation}\label{introduccion cita 1}
\partial_t\tilde{\bm{\omega}}-\nu\Delta\tilde{\bm{\omega}}=NL[\tilde{\bm{\omega}},\tilde{\bm{v}}]+L[\bm{\omega}_0,\tilde{\bm{\omega}}]+\bm F[\bm{\omega}_0,\bm v_0],
\end{equation}
where
\[
NL[\tilde{\bm{\omega}},\tilde{\bm{v}}]=-(\tilde{\bm v}\cdot\nabla)\tilde{\bm\omega}+(\tilde{\bm\omega}\cdot\nabla)\tilde{\bm v},
\]
and
\[
L[\bm{\omega}_0,\tilde{\bm{\omega}}]=-(\tilde{\bm v}\cdot\nabla)\bm\omega_0-(\bm v_0\cdot\nabla)\tilde{\bm\omega}+(\tilde{\bm\omega}\cdot\nabla)\bm v_0+(\bm\omega_0\cdot\nabla)\tilde{\bm v}.
\]
The term $\bm F[\bm{\omega}_0,\bm v_0]$ is the forcing generated by the discrepancy between the ansatz and the exact Navier--Stokes dynamics, and its structure depends explicitly on the geometry of the filament. A key point is that the dominant singular contributions in the forcing cancel out, leaving only lower-order error terms. This cancellation makes it possible to show that the forcing remains perturbatively small in the relevant regime. The proof of the existence of $\tilde{\vo}$ is based on Duhamel's formula and on a fixed point argument in time-weighted Morrey spaces, following the framework of Giga and Miyakawa \cite{GigaMiyakawa}. As explained in Remark \ref{rem: asymptotics fontelos vega}, by looking at the first term of this scheme one recovers the asymptotics exhibited in Theorem \ref{thm: fontelos vega}.

One of the main contributions of this work is to provide a precise description of $\bm F[\vo_0,\vv_0]$, with particular emphasis on its dependence on the geometric quantities of the filament. More precisely, we identify the terms which may depend on torsion and show that the perturbative estimates still hold in a regime where the curvature remains bounded while the torsion is allowed to be large. This sharper description is essential for applying the existence theorem uniformly in the large-torsion regime to be considered later.

Furthermore, this refined understanding has consequences beyond the existence result itself. In Section \ref{sec: moving filament}, we apply these estimates to the specific case of $\XX$ being the Hasimoto soliton. Hasimoto showed in \cite{Hasimoto72} that the binormal flow \eqref{definicion binormal flow}, which after a suitable time rescaling can be written as
\[
\XX_t=\kappa \bm b,
\]
can be transformed into the cubic nonlinear Schrödinger equation
\[
\frac{1}{i}\psi_t=\psi_{ss}+\frac12\big(|\psi|^2+A(t)\big)\psi.
\]
Here $\psi$ is defined as the following wave function
\[
\psi(s,t)=\kappa(s,t)\exp\left(i\int_0^s\tau(s',t)\,{\rm d}s'\right),
\]
where $\tau$ denotes the torsion of the curve. In this setting, the Hasimoto soliton corresponds to a curvature profile of the form
\[
\kappa_\lambda(s,t)=2\lambda\,\operatorname{sech}\big(\lambda(s-2\tau_0 t)\big),
\]
with constant torsion $\tau_0$ and maximum curvature $2\lambda$. Thus the curvature is localized on a length scale $O(\lambda^{-1})$, while its center propagates along the filament with speed $2\tau_0$ in the rescaled time variable. We will choose $\lambda=1$ and
\[
\tau_0\sim \frac{1}{(\Gamma/\nu)\nu T|\log(\nu T)|}\gg1,
\]
as in \eqref{curvature and torsion order}, so that, after undoing the time rescaling associated with the physical binormal flow, the center of the curvature bump undergoes an order-one displacement during an admissible time interval.

Remarkably, for the Hasimoto soliton, the estimates entering the existence result are uniform with respect to the torsion parameter. This relies on two key facts. First, for the Hasimoto soliton, the Biot--Savart estimate for $\vv^*$ is independent of the torsion (see Section~\ref{subsec:Uniformity of the estimates in the torsion parameter} for details, in particular Lemma \ref{cancelacion en el soliton para A}). Second, after introducing the ansatz $\vo_0$, the perturbation equation \eqref{introduccion cita 1} is also independent of the torsion due to a cancellation between terms produced by the time derivative of $\vo_0$ and the vortex stretching term (see \eqref{segunda cancelación temporal}). This cancellation is used in Corollary \ref{Corolario hasimoto} to prove that the constant $\mathcal{C}_F(\XX,\Gamma,\nu T)$ in \eqref{norma morrey introduccion} and \eqref{norma velocidad morrey intro} does not depend on the torsion parameter of the Hasimoto soliton. Consequently, Theorem \ref{Teorema vorticidad intro} remains valid even in this large-torsion regime with estimates that do not depend on the size of the torsion. Moreover, the cancellation of the stretching term exhibited in \eqref{segunda cancelación temporal} is generic and not just for the case of Hasimoto soliton.

Finally, as shown in Section \ref{sec: moving filament}, the Navier--Stokes solution constructed around such a curve contains a localized portion of kinetic energy, associated with the binormal component of the velocity, which remains concentrated inside a moving tubular region, the so-called physical core defined in \eqref{core region label 0} and \eqref{label-core-region}, that undergoes an order-one displacement over the time interval $(0,T)$. See \cite{Ricca1992} for a physical interpretation of the conservation laws of binormal flow. Also, this type of phenomenon appears in the ill- posedness results of the IVP of the cubic NLS proved in \cite{KPV2001} and later extended in \cite{JS2012} to the Schrödinger map equation using Kida solutions.

The presence of soliton-like structures in viscous fluids was already observed in the rotating-tank experiments of Hopfinger, Browand and Gagné \cite{HopfingerBrowandGagne}. In their experiment, turbulence was generated by an oscillating grid in a rotating container. The resulting flow developed long concentrated vortices with helical distortions propagating along their cores. These waves were interpreted as vortex solitary waves and were found to be well described by Hasimoto's vortex-soliton theory. This is consistent with the recent numerical experiments of Sterkers and Krstulovic \cite{SterkersKrstulovic2026}. They performed three-dimensional Navier--Stokes simulations of viscous vortex filaments and observed the propagation of Kelvin waves and Hasimoto-type solitons along concentrated vortex cores. In a high-Reynolds-number regime, $\Gamma/\nu\gg 1$, their simulations show that the observed soliton speed is in good agreement with the prediction obtained from Hasimoto's theory through the localized-induction approximation.
\normalcolor
\subsection{Plan of the paper}

In Section \ref{sec: the local frame}, we introduce the geometric framework used throughout the paper. We define the tubular neighborhood around the filament, construct the associated local coordinate system, and describe the parallel frame formulation of the binormal flow. These coordinates allow us to express the Navier--Stokes equations in a form adapted to the filament geometry.

In Section \ref{sec: biot savart integral}, we study the Biot--Savart law associated with the leading-order vorticity ansatz. We derive an expansion for the induced velocity field, showing that it consists of three main contributions: the local Lamb--Oseen velocity, the velocity $\XX_t$ of the reference curve, and a remainder term whose dependence on the geometry of the filament is made explicit.

In Section \ref{sec: Navier Stokes curvilinear}, we insert the asymptotic ansatz into the Navier--Stokes equations and derive the evolution equation satisfied by the perturbation term. A detailed analysis of the forcing term is carried out, emphasizing the cancellation of the dominant singular terms and identifying precisely how curvature and torsion enter into the error.

In Section \ref{sec: existence result}, we establish the main existence theorem by solving the perturbative equation in Morrey spaces. Using the framework of Giga and Miyakawa, we obtain uniform estimates for the correction term and prove that the full solution remains close to the leading-order filament profile for sufficiently short times.

Finally, in Section \ref{sec: moving filament}, we apply our refined estimates to the explicit Hasimoto soliton solution of the binormal flow. By considering a regime of bounded curvature and large torsion, we construct a filament whose localized curvature bump travels an order-one distance during the time interval under consideration. This provides an explicit example in which the corresponding Navier--Stokes solution exhibits a quantitatively observable displacement of a localized portion of kinetic energy while remaining within the perturbative regime guaranteed by our existence theory.
\section{The local frame}\label{sec: the local frame}

In this Section, we introduce the geometric framework used throughout the paper. We define the tubular neighborhood around the filament, construct the associated local coordinate system and describe the parallel frame formulation of the binormal flow. These coordinates allow us to express the Navier--Stokes equations in a form adapted to the geometry of the filament.
\subsection{The ``filament'' and its tubular neighborhood}\label{subsec: filament and tubular neighborhood}
Let $\XX(s,t)\subset\mathbb{R}^3$ be a smooth curve parametrized by arc-length, and let $\TT(s,t)=\partial_s\XX(s,t)$. We shall sometimes call $\XX$ the filament, by a slight abuse of terminology. More precisely, $\XX$ is the reference curve on which the moving frame is constructed, and the vorticity profile will later be localized around it.
We assume that $\XX$ evolves according to the binormal flow which can be expressed by the following two equivalent formulations:
\begin{subequations}\label{binormal_flow}
\begin{align}
\bm{\chi}_t&=\mathsf{c}\left(\bm{\chi}_s\wedge\bm{\chi}_{ss}\right),\label{binormal_flow_chi}\\
\bm{T}_t&=\mathsf{c}\left(\bm{T}\wedge\bm{T}_{ss}\right),\label{binormal_flow_T}
\end{align}
\end{subequations}
where, for notational convenience, we set $\mathsf{c}=-\frac{\Gamma}{4\pi}\log(\sqrt{\nu t})$.
We assume throughout this section that $\XX$ is sufficiently regular for the geometric quantities below to be well defined. The precise regularity needed in the desingularization of the Biot--Savart integral will be discussed in Sections \ref{sec: biot savart integral} and \ref{sec: moving filament}.
Let us define
\[
d_{\rm arc}(s,s') :=
\begin{cases}
|s-s'|, & \text{if the curve is open},\\[6pt]
\displaystyle \min\{\,|s-s'|,\; L-|s-s'|\,\}, & \text{if the curve is closed of total length } L.
\end{cases}
\]
\begin{defn}[Uniform chord-arc condition]\label{chord_arc_condition}
We say that the arc-length parametrized curve $\XX(s,t)$ satisfies the \emph{uniform chord-arc condition} on $[0,T]$ if there exists $c_0>0$, independent of $t$, such that
\begin{equation}\label{uniform_chord_arc_condition}
|\XX(s,t)-\XX(s',t)|\geq c_0 d_{\rm arc}(s,s')
\end{equation}
for all $s,s'$ and all $t\in[0,T]$.
\end{defn}
For such curves, provided the curvature is bounded, we introduce the following regular tubular neighborhood.
\begin{defn}[Regular tubular neighborhood]\label{regular_tubular_neighborhood}
Let $\XX(s,t)\subset\mathbb{R}^3$ be a smooth curve parametrized by arc-length, and let $R>0$. We define the tubular neighborhood of radius $R$ around $\XX(\cdot,t)$ by
\[
\mathcal{T}_R(t):=\left\{\xx\in\mathbb{R}^3:\exists s \text{ such that } |\xx-\XX(s,t)|<R\right\}.
\]
We say that $\mathcal{T}_R(t)$ is a \emph{regular tubular neighborhood} if, for every $\xx\in\mathcal{T}_R(t)$, there exists a unique point $\XX(s_0,t)$ on the curve such that
\[
|\xx-\XX(s_0,t)|=\min_s|\xx-\XX(s,t)|.
\]
\end{defn}
In this case, the vector $\xx-\XX(s_0,t)$ is orthogonal to the curve at the closest point, that is,
\[
(\xx-\XX(s_0,t))\cdot\TT(s_0,t)=0.
\]
A sufficient condition for $\mathcal{T}_R(t)$ to be a regular tubular neighborhood for every $t\in[0,T]$ is that $R$ satisfies
\begin{equation}\label{definicion radio entorno tubular}
R:=\frac{1}{4}\min\left\{\frac{1}{\kappa_{\max}},\frac{1}{2}\inf_{\substack{t\in[0,T]\\ d_{\rm arc}(s,s')>\pi/\kappa_{\max}}}|\XX(s,t)-\XX(s',t)|\right\},
\end{equation}
where
\[
\kappa_{\max}:=\sup_{t\in[0,T]}\sup_s|\partial_s\TT(s,t)|.
\]
The definition of $R$ combines a local condition based on the curvature and a global non-intersection condition, uniformly in time. The numerical factor $1/4$ is not optimal, it is fixed only to keep a convenient margin in the estimates below.

\subsection{Parallel frame and local coordinates}\label{subsec: Parallel frame and local coordinates}
We now introduce the local frame attached to the reference curve $\XX$ introduced in Subsection \ref{subsec: filament and tubular neighborhood}. This choice is important because, when the Navier--Stokes equations are written in the corresponding local coordinates, additional terms appear due to the motion and rotation of the frame itself.
In the following geometric computations, time is fixed and we omit the dependence on $t$ whenever no confusion can arise.
The classical frame used in the study of space curves is the \emph{Frenet--Serret frame} $(\TT,\bm n,\bm b)$, formed by the tangent $\TT$, normal $\bm{n}$ and binormal $\bm{b}$ vectors, and defined by
\begin{align}\label{frenet_serret_frame}
\begin{bmatrix}
\TT\\
\bm n\\
\bm b
\end{bmatrix}_s
=
\begin{bmatrix}
0 & \kappa & 0\\
-\kappa & 0 & \tau\\
0 & -\tau & 0
\end{bmatrix}
\begin{bmatrix}
\TT\\
\bm n\\
\bm b
\end{bmatrix}.
\end{align}
In this frame, any point $\xx$ in a tubular neighborhood $\mathcal T_R$ of the curve can be written as
\begin{equation}\label{x en frenet serret}
\xx=\XX(s)+r\cos\theta\,\bm n(s)+r\sin\theta\,\bm b(s).
\end{equation}
Although the Frenet--Serret frame is geometrically natural and orthonormal wherever it is defined, it has some drawbacks for our purposes. First, it is not defined at points where $\kappa=0$, since the normal vector is given by $\bm n=\TT_s/|\TT_s|$. Second, the associated tubular coordinates inherit the torsional rotation of the normal plane. In particular, if we compute the tangent vector of the coordinate lines $(s,r,\theta)$ in \eqref{x en frenet serret}, we get
\begin{align*}
\partial_s\xx&=(1-r\kappa\cos\theta)\TT+r\tau(\cos\theta\,\bm{b}-\sin\theta\,\bm{n}),\\
\partial_r\xx&=\cos\theta\,\bm{n}+\sin\theta\,\bm{b},\\
\partial_\theta\xx&=-r\sin\theta\,\bm{n}+r\cos\theta\,\bm{b}.
\end{align*}
As can be seen from these expressions, the coordinate vectors are not mutually orthogonal in general. Consequently, when differential operators, such as those appearing in \eqref{vorticity_NS}, are written in these coordinates, additional terms involving the torsion appear. This is the main reason why the Frenet--Serret frame is not the most convenient choice for the computations below.
For this reason, we use instead the \emph{parallel frame }$\{\TT,\ea_1,\ea_2\}$. This is an orthonormal frame along the reference curve such that $\ea_1$ and $\ea_2$ span the normal plane and do not rotate within that plane as $s$ varies. Equivalently, the $s$-derivatives of $\ea_1$ and $\ea_2$ have only tangential components:
\[
\TT_s=\alpha\ea_1+\beta\ea_2,\qquad (\ea_1)_s=-\alpha\TT,\qquad (\ea_2)_s=-\beta\TT.
\]
Using this frame, we introduce cylindrical coordinates $(r, \varphi, s)$ in the normal plane by writing
\begin{equation}\label{cyl_coor}
\xx=\XX(s)+r\ea_r,\qquad \ea_r=\cos\varphi\,\ea_1+\sin\varphi\,\ea_2.
\end{equation}
Then $\partial_s\xx$ is parallel to $\TT$, while $\partial_r\xx$ and $\partial_\varphi\xx$ lie in the normal plane. Hence the coordinate system $(r,\varphi,s)$ is orthogonal and its scale factors are
\begin{equation}\label{scale_factors}
h^r=1,\qquad h^\varphi=r,\qquad h:=h^s=1-r(\alpha\cos\varphi+\beta\sin\varphi).
\end{equation}
This is the main advantage over the Frenet--Serret frame. In the Frenet--Serret coordinates, changing $s$ also produces a rotation in the angular direction whenever the torsion is nonzero. In the parallel frame, changing $s$ moves the coordinate lines only in the tangential direction. This makes the resulting expressions for the differential operators considerably simpler.
The relation with the Frenet--Serret frame is encoded in the complex curvature
\[
\psi(s,t)=\alpha(s,t)+i\beta(s,t)=\kappa(s,t)\exp\left(i\int_0^s\tau(s',t)\,{\rm d}s'\right).
\]
If $\theta_0(s,t)=\int_0^s\tau(s',t)\,{\rm d}s'$, then
\begin{equation}\label{alpha_beta_trig}
\alpha=\kappa\cos\theta_0,\qquad \beta=\kappa\sin\theta_0.
\end{equation}
With the corresponding choice of parallel frame, the Frenet--Serret angle $\theta$ and the parallel-frame angle $\varphi$ satisfy
\[
\varphi=\theta+\theta_0.
\]
Thus the use of the parallel frame amounts to choosing an angular variable for which the coordinate vector $\partial_s\xx$ has no component in the angular direction.
We now recall how this frame behaves under the binormal flow introduced in \eqref{binormal_flow}. Since $\TT_s=\alpha\ea_1+\beta\ea_2$, it follows from \eqref{binormal_flow_chi} that
\[
\XX_t=\mathsf{c}\,\TT\wedge\TT_s=\mathsf{c}(\alpha\ea_2-\beta\ea_1).
\]
Moreover, by \eqref{binormal_flow_T}, using the defining relations of the parallel frame we obtain
\[
\TT_t=\mathsf{c}(-\beta_s\ea_1+\alpha_s\ea_2).
\]
By orthogonality of the frame, its time evolution is therefore given by
\begin{align}\label{t_derivative_parallel_frame}
\begin{bmatrix}\ea_1\\ \ea_2\\ \TT\end{bmatrix}_t=\begin{bmatrix}0 & \gamma & \mathsf{c}\beta_s\\ -\gamma & 0 & -\mathsf{c}\alpha_s\\ -\mathsf{c}\beta_s & \mathsf{c}\alpha_s & 0\end{bmatrix}\begin{bmatrix}\ea_1\\ \ea_2\\ \TT\end{bmatrix},
\end{align}
where $\gamma=\langle(\ea_1)_t,\ea_2\rangle=-\langle(\ea_2)_t,\ea_1\rangle$.
In this setting, after a suitable time reparametrization to absorb $\mathsf{c}$, the Hasimoto transform converts the binormal flow equation into
\begin{equation}\label{cubica schrodinger}
\frac{1}{i}\psi_t=\psi_{ss}+\frac12\bigl(|\psi|^2+A(t)\bigr)\psi.
\end{equation}
with $A$ being a real function. This equation exhibits simple transformation laws, such as Galilean transformations, and admits classical shape-preserving solutions, including circles, helices and solitons (see \cite{Hasimoto72}).
\subsection{Derivatives of the local coordinates}
Now, we compute the derivatives of the local coordinates introduced above. These identities will be used later to write the differential operators appearing in the Navier--Stokes equations in the coordinates adapted to the reference curve.
Let $\{x_i\}_{i=1}^3$ be the Cartesian coordinate system and let $\{\bar{\ea}_i\}_{i=1}^3$ be the associated basis. As in \eqref{cyl_coor}, every point $\xx\in\mathcal{T}_R$ can be written in local cylindrical coordinates as
\[
\xx=\XX(s,t)+r\ea_r(s,\varphi,t).
\]
Differentiating this identity with respect to the Cartesian variables, and using the scale factors computed in \eqref{scale_factors}, we obtain
\begin{equation}\label{eq:derivadas_coordenadas}
\frac{\partial r}{\partial x_i}=r_i=\bar{\ea}_i\cdot\ea_r,\qquad
\frac{\partial \varphi}{\partial x_i}=\varphi_i=\frac{1}{r}(\bar{\ea}_i\cdot\ea_\varphi),\qquad
\frac{\partial s}{\partial x_i}=s_i=\frac{1}{h}(\bar{\ea}_i\cdot\TT),
\end{equation}
where $h$ denotes the scale factor of the variable $s$ defined in \eqref{scale_factors}. We shall also need the time derivatives of the local coordinates. These derivatives are always understood at fixed Cartesian point $\xx$. In other words, when $\xx$ in \eqref{cyl_coor} is kept fixed, the local coordinates $r=r(\xx,t)$, $\varphi=\varphi(\xx,t)$ and $s=s(\xx,t)$ depend on time. Differentiating \eqref{cyl_coor} with respect to $t$, keeping $\xx$ fixed, we deduce
\[
0=\XX_t+r_t\ea_r+r\varphi_t\ea_\varphi+s_t h\TT+r\partial_t\ea_r.
\]
Here $\partial_t\ea_r$ denotes the explicit time derivative of the frame at fixed $(s,\varphi)$, whereas $r_t$, $\varphi_t$ and $s_t$ are the time derivatives of the local coordinates at fixed $\xx$. Then, taking scalar products with $\ea_r$, $\ea_\varphi$ and $\TT$, respectively, we obtain
\begin{equation}\label{eq:derivadas_tiempo_coordenadas}
r_t=-\XX_t\cdot\ea_r,\qquad
\varphi_t=-\frac{1}{r}(r\partial_t\ea_r+\XX_t)\cdot\ea_\varphi,\qquad
s_t=-\frac{1}{h}(r\partial_t\ea_r+\XX_t)\cdot\TT.
\end{equation}
Since the binormal velocity satisfies $\XX_t\cdot\TT=0$, the last identity reduces to
\[
s_t=-\frac{r}{h}\partial_t\ea_r\cdot\TT.
\]
\section{Biot-Savart integral}\label{sec: biot savart integral}
This section is devoted to the study of the Biot--Savart law associated with the leading-order vorticity ansatz for \eqref{vorticity_NS}. We work in the tubular neighborhood and in the local coordinates constructed in Section \ref{sec: the local frame}. The goal is to desingularize the Biot--Savart integral and to identify the main contributions to the induced velocity field: the local Lamb--Oseen velocity, the binormal-flow term and a remainder whose dependence on the geometry of the reference curve is made explicit.
We begin by rewriting \eqref{definicion biot savart} in the local coordinates introduced in Subsection \ref{subsec: Parallel frame and local coordinates}. The computation is organized in two steps. First, for fixed perpendicular coordinates $(r,\varphi)$, we estimate the integral in the tangential variable $s$, which is carried out independently in Lemma \ref{lem:BS-in-s}. We then integrate the resulting expression in the perpendicular variables.

\subsection{Vorticity ansatz}
As mentioned in the introduction, the leading-order vorticity ansatz for \eqref{vorticity_NS} is given by
\begin{align}\label{lamb_oseen_solution}
\vo_0(\xx,t)=\Omega^\eta(r,t)\TT(s,t),
\end{align}
where $(r,\varphi,s)$ are the local cylindrical coordinates around the reference curve defined in \eqref{cyl_coor}. Here
\[
\Omega(r,t):=\frac{\Gamma}{4\pi\nu t}e^{-\frac{r^2}{4\nu t}},
\qquad
\Omega^\eta(r,t):=\Omega(r,t)\eta_R(r),
\]
and $\eta_R\in C^{\infty}$ is a cutoff function satisfying
\begin{equation}\label{definicion cutoff}
\eta_R(r)=
\begin{cases}
0,& \mbox{if } r\geq R,\\
1,& \mbox{if } r\leq R/2.
\end{cases}
\end{equation}
Throughout this section, time is fixed and plays the role of a parameter, and we therefore omit the dependence of the reference curve and of the frame on $t$ whenever no confusion can arise. The role of the cutoff is to localize the support of $\vo_0$ inside the regular tubular neighborhood of radius $R$ defined in \eqref{definicion radio entorno tubular}. In particular, for every $\xx\in\mathcal{T}_R$ there exists a unique closest point on the reference curve, and hence a unique value of $s$ such that
\[
s=\operatorname*{arg\,min}_{s'\in\RR} |\xx-\XX(s')|.
\]
Sometimes, for convenience, we will use the self-similar variable introduced in \eqref{def: selfsimilar variable}:
\[
\rho=\frac{r}{\sqrt{\nu t}}.
\]
We define the orthogonal displacement vector from the reference curve to the point $\xx$ by
\begin{equation}\label{def: orthogonal vector}
\hx=\xx-\XX(s)=r\ea_r(s)=r\cos(\varphi)\ea_1(s)+r\sin(\varphi)\ea_2(s),
\end{equation}
where $\ea_r$ is the radial direction in local cylindrical coordinates and $\ea_1$ and $\ea_2$ are the orthogonal vectors to $\TT$ that form the parallel frame (see Section \ref{sec: the local frame} for further details).
The velocity field $\vv_0$ induced by $\vo_0$ is obtained by inserting this ansatz into the Biot--Savart law \eqref{definicion biot savart}. Written in the local coordinates of the reference curve, this gives
\begin{equation}\label{eq:BS-local-coordinates}
\begin{aligned}
\vv_0(\xx)
&=\frac{1}{4\pi}\int_{\RR^3}\vo_0(\yy)\wedge\frac{\xx-\yy}{|\xx-\yy|^3}\,{\rm d}\yy\\
&=\frac{1}{4\pi}\int_{\RR^2}\Omega^{\eta}(|\hy|)
\left(\int_{\RR}\TT(s)\wedge\frac{\XX(s_0)+\hx-\XX(s)-\hy(r,\varphi,s)}{|\XX(s_0)+\hx-\XX(s)-\hy(r,\varphi,s)|^3}\,h(r,\varphi,s)\,{\rm d}s\right)\,{\rm d}\hy,
\end{aligned}
\end{equation}
where $\xx=\XX(s_0)+\hx$, $\yy=\XX(s)+\hy(r,\varphi,s)$ and $\hy(r,\varphi,s)=r\ea_r(s,\varphi)$.
\subsection{Integral in the tangential variable}
Before computing $\vv_0$, we first study the inner integral in \eqref{eq:BS-local-coordinates}, obtained by fixing the transverse coordinate $\hy$. This term can be interpreted as the contribution to the velocity induced by the curve parallel to $\XX$ corresponding to that fixed transverse coordinate.
The estimate of this tangential integral depends on a quantity measuring the variation of $\TT'$ along the reference curve. For $\xx=\XX(s_0)+\hx$, with $\hx\perp\TT(s_0)$, we set
\begin{equation}\label{s_limit_condition}
l=\frac{1}{5\|\TT'\|_{\infty}},\qquad A(l,|\hx|)=\int_{s_0-l}^{s_0+l}\int_{s_0}^s\frac{\TT'(\xi)-\TT'(s_0)}{\xi-s_0}g(\xi,s)\,{\rm d}\xi{\rm d}s,
\end{equation}
where
\[
g(\xi,s)=\frac{(\xi-s_0)^2}{\left(|\hx|^2+|s-s_0|^2\right)^{3/2}}.
\]
\begin{lem}\label{lem:BS-in-s}
Let $\XX:\RR\to\RR^3$ be an open curve parametrized by arclength $s$, with $\XX\in C^{2,\alpha}$ for some $\alpha>0$. Assume that $\|\XX''\|_\infty=\|\TT'\|_{\infty}<\infty$. Assume also that the radius $R$ of the tubular neighborhood given by \eqref{definicion radio entorno tubular} is well defined and that there exists $c_0>0$ such that the chord-arc condition \eqref{uniform_chord_arc_condition_intro} holds. Then, for any $\xx\in\mathcal{T}_{R}$, we can write $\xx=\XX(s_0)+\hx$, with $\hx\perp\TT(s_0)$, and the following expression holds:
\[
\int_{\RR}\frac{\TT(s)\wedge(\xx-\XX(s))}{|\xx-\XX(s)|^3}\,{\rm d}s=2\frac{\TT(s_0)\wedge\hx}{|\hx|^2}-\TT(s_0)\wedge\TT'(s_0)\log|\hx|+\bm{\tilde{v}}^*(\xx,\XX).
\]
Furthermore, defining $A(l,|\hx|)$ as in \eqref{s_limit_condition}, the remainder satisfies
\[
|\bm{\tilde{v}}^*(\xx,\XX)|\leq C\left(|A(l,|\hx|)|+\|\TT'\|_\infty\left(1+|\log(\|\TT'\|_\infty)|\right)\right),
\]
for some constant $C>0$ independent of $\xx$ and $\XX$.
\end{lem}
\begin{rem}\label{rem:Dini-condition}
	The regularity assumption $\XX\in C^{2,\alpha}$ ensures the pointwise boundedness of $A$. More precisely, since $\TT'=\XX''$ is Hölder continuous, it satisfies a Dini-type condition. Indeed, changing the order of integration and using the explicit form of $g$ we obtain
	\[
	|A(l,|\hx|)|\leq \frac12\int_{s_0-l}^{s_0+l}\left|\frac{\TT'(\xi)-\TT'(s_0)}{\xi-s_0}\right|\,{\rm d}\xi.
	\]
	Therefore, it is enough to require the Dini-type bound
	\begin{equation}\label{eq:Dini-condition}
		\sup_{s_0\in\RR}\int_{s_0-l}^{s_0+l}\left|\frac{\TT'(s)-\TT'(s_0)}{s-s_0}\right|\,{\rm d}s<\infty.
	\end{equation}
	However, the estimate above is obtained after taking absolute values and may therefore lose cancellations present in the signed quantity $A$. In particular, the Dini-type bound may depend on the torsion of the curve. This is why, in Section \ref{sec: moving filament}, we keep track of the more precise form of $A$ given in \eqref{s_limit_condition}.
\end{rem}
\begin{rem}
Let us define the rescaled curve
\[
\XX_\lambda(s)=\lambda^{-1}\XX(\lambda s).
\]
Then
\[
\TT_\lambda(s)=\TT(\lambda s)\qquad\mbox{and}\qquad \TT_\lambda'(s)=\lambda\TT'(\lambda s).
\]
Accordingly, the quantity $A$ scales as
\[
A_\lambda(l,r_0)=\lambda A(\lambda l,\lambda r_0),
\]
i.e. has the natural scaling of the curvature. Therefore, the estimate for the remainder has the correct scaling:
\begin{equation}
|\bm{\tilde v}^*(\xx,\XX_\lambda)|\leq C\lambda(1+|\log{\lambda}|).
\end{equation}
where $C$ is independent of $\lambda$.
\end{rem}
\begin{proof}[Proof of Lemma \ref{lem:BS-in-s}]
Throughout this proof, $C>0$ denotes a universal constant, independent of the geometric properties of the reference curve, which may change from line to line. First, without loss of generality, we set $s_0=0$ and define $\XX_0=\XX(0)$ and $\TT_0=\TT(0)$. Then, we split the integral and isolate the region where the main contribution comes from:
\begin{equation}\label{main_part_BS}
\int_{\RR}\frac{\TT(s)\wedge(\xx-\XX(s))}{|\xx-\XX(s)|^3}\,{\rm d}s=\int_{-l}^l\frac{\TT(s)\wedge(\xx-\XX(s))}{|\xx-\XX(s)|^3}\,{\rm d}s+\int_{\RR\setminus[-l, l]}\frac{\TT(s)\wedge(\xx-\XX(s))}{|\xx-\XX(s)|^3}\,{\rm d}s.
\end{equation}
The main contribution comes from the local term above, so let us rewrite it in a more convenient form. To do this, we notice that
\begin{align*}
|\xx-\XX(s)|^2&=\left|\hx-\int_0^{s}\TT(\xi)\,{\rm d}\xi\right|^2\\
&=\left|\hx-\TT_0s-\int_0^{s}\left(\TT(\xi)-\TT_0\right)\,{\rm d}\xi\right|^2\\
&=|\hx|^2+s^2-2\,\hx\cdot\left(\int_0^{s}\left(\TT(\xi)-\TT_0\right)\,{\rm d}\xi\right)+2s\,\TT_0\cdot\left(\int_0^{s}\left(\TT(\xi)-\TT_0\right)\,{\rm d}\xi\right)\\
&\quad+\left|\int_0^{s}\left(\TT(\xi)-\TT_0\right)\,{\rm d}\xi\right|^2\\
&=|\hx|^2+s^2+F(\hx,\XX(s),s).
\end{align*}
Thus,
\[
|F|\leq|\hx|s^2\|\TT'\|_{\infty}+|s|^3\|\TT'\|_{\infty}+\frac{s^4}{4}\|\TT'\|^2_{\infty}.
\]
Moreover, since $|s|\leq l$ and $\xx\in\mathcal{T}_{R}$, we have
\begin{align}\label{estimacion 1}
\left|\frac{F}{|\hx|^2+s^2}\right|\leq|\hx|\|\TT'\|_{\infty}+|s|\|\TT'\|_{\infty}+\frac{s^2}{4}\|\TT'\|^2_{\infty}\leq\frac{1}{2}.
\end{align}
Then,
\begin{align*}
\frac{1}{|\xx-\XX(s)|^3}&=\frac{1}{\left(|\hx|^2+s^2+F\right)^{3/2}}\\
&=\frac{1}{\left(|\hx|^2+s^2\right)^{3/2}}+\left[\frac{1}{\left(|\hx|^2+s^2+F\right)^{3/2}}-\frac{1}{\left(|\hx|^2+s^2\right)^{3/2}}\right]\\
&=\frac{1}{\left(|\hx|^2+s^2\right)^{3/2}}+R(\hx,\XX(s),s),
\end{align*}
where
\[
R(\hx,\XX(s),s)=\frac{\left(1+\frac{F}{|\hx|^2+s^2}\right)^{-3/2}-1}{\left(|\hx|^2+s^2\right)^{3/2}}.
\]
The mean value theorem together with \eqref{estimacion 1} yields, for every $|\zeta|\leq 1/2$,
\[
|1-(1+\zeta)^{-3/2}|\leq \sup_{|\xi|\leq\frac12}\frac{3}{2}(1+\xi)^{-5/2}|\zeta|\leq 6\sqrt{2}\,|\zeta|.
\]
Therefore,
\begin{equation}\label{equation noname 1}
\begin{aligned}
|R(\hx,s)|&\leq6\sqrt{2}\;\frac{\left|\frac{F}{|\hx|^2+s^2}\right|}{\left(|\hx|^2+s^2\right)^{3/2}}\\
&\leq6\sqrt{2}\;\frac{|\hx|\|\TT'\|_{\infty}+|s|\|\TT'\|_{\infty}+\frac{s^2}{4}\|\TT'\|^2_{\infty}}{\left(|\hx|^2+s^2\right)^{3/2}}.
\end{aligned}
\end{equation}
Now, back to \eqref{main_part_BS}, we write
\begin{align*}
\int_{\RR}\frac{\TT(s)\wedge(\xx-\XX(s))}{|\xx-\XX(s)|^3}\,{\rm d}s
=&\int_{-l}^l\frac{\TT(s)\wedge(\xx-\XX(s))}{\left(|\hx|^2+s^2\right)^{3/2}}\,{\rm d}s+\int_{-l}^l\left[\TT(s)\wedge(\xx-\XX(s))\right]R(\hx,s)\,{\rm d}s\\
&+\int_{\RR\setminus[-l, l]}\frac{\TT(s)\wedge(\xx-\XX(s))}{|\xx-\XX(s)|^3}\,{\rm d}s.
\end{align*}
We split the first integral above as
\begin{align*}
\int_{-l}^l\frac{\TT(s)\wedge(\xx-\XX(s))}{\left(|\hx|^2+s^2\right)^{3/2}}\,{\rm d}s
&=\int_{-l}^l\frac{\TT(s)\wedge\hx}{\left(|\hx|^2+s^2\right)^{3/2}}\,{\rm d}s+\int_{-l}^l\frac{\TT(s)\wedge(\XX_0-\XX(s))}{\left(|\hx|^2+s^2\right)^{3/2}}\,{\rm d}s\\
&=I_1+I_2.
\end{align*}
For $I_1$ we proceed in the following way:
\begin{align*}
I_1&=\TT(0)\wedge\hx\int_{-l}^l\frac{{\rm d}s}{\left(|\hx|^2+s^2\right)^{3/2}}+\int_{-l}^l\frac{\left(\TT(s)-\TT(0)\right)\wedge\hx}{\left(|\hx|^2+s^2\right)^{3/2}}\,{\rm d}s\\
&=\TT_0\wedge\hx\int_{\RR}\frac{{\rm d}s}{\left(|\hx|^2+s^2\right)^{3/2}}-\TT_0\wedge\hx\int_{\RR\setminus[-l,l]}\frac{{\rm d}s}{\left(|\hx|^2+s^2\right)^{3/2}}+\int_{-l}^l\frac{\left(\TT(s)-\TT_0\right)\wedge\hx}{\left(|\hx|^2+s^2\right)^{3/2}}\,{\rm d}s\\
&=2\frac{\TT_0\wedge\hx}{|\hx|^2}+E_1,
\end{align*}
where
\begin{align*}
|E_1|&\leq\frac{2}{|\hx|}\int_{l/|\hx|}^{\infty}\frac{{\rm d}\sigma}{\left(1+\sigma^2\right)^{3/2}}+2\|\TT'\|_{\infty}\int_0^{l/|\hx|}\frac{\sigma{\rm d}\sigma}{\left(1+\sigma^2\right)^{3/2}}\\
&\leq 2\left(\frac{|\hx|}{l^2}+\|\TT'\|_{\infty}\right)\leq C\|\TT'\|_{\infty}.
\end{align*}
To compute $I_2$, first we notice that
\begin{equation}\label{equation noname 2}
\XX_0-\XX(s)=-s\TT(s)+\int_0^s\TT'(\xi)\xi\,{\rm d}\xi=-s\TT(s)+\frac{1}{2}\TT'_0s^2+\int_0^s\left[\TT'(\xi)-\TT'_0\right]\xi\,{\rm d}\xi.
\end{equation}
Thus,
\begin{align*}
I_2=&\;\frac{1}{2}\TT(0)\wedge\TT'(0)\int_{-l}^l\frac{s^2\,{\rm d}s}{\left(|\hx|^2+s^2\right)^{3/2}}+\int_{-l}^l\TT(s)\wedge\left(\int_0^s\left[\TT'(\xi)-\TT'(0)\right]\xi\,{\rm d}\xi\right)\frac{{\rm d}s}{\left(|\hx|^2+s^2\right)^{3/2}}\\
&+\frac{1}{2}\int_{-l}^l\frac{\left[\TT(s)-\TT(0)\right]\wedge\TT'(0)s^2}{\left(|\hx|^2+s^2\right)^{3/2}}\,{\rm d}s\\
=&\;\frac{1}{2}\TT(0)\wedge\TT'(0)\int_{-l}^l\frac{s^2\,{\rm d}s}{\left(|\hx|^2+s^2\right)^{3/2}}+E_2.
\end{align*}
Now, we compute the main term above:
\begin{align*}
&\frac{1}{2}\TT(0)\wedge\TT'(0)\int_{-l}^l\frac{s^2\,{\rm d}s}{\left(|\hx|^2+s^2\right)^{3/2}}\\
&=\TT(0)\wedge\TT'(0)\left[\log\left(\frac{l}{|\hx|}+\sqrt{1+\left(\frac{l}{|\hx|}\right)^2}\right)-\frac{\frac{l}{|\hx|}}{\sqrt{1+\left(\frac{l}{|\hx|}\right)^2}}\right]\\
&=\TT(0)\wedge\TT'(0)\left[\log\left(\frac{l}{|\hx|}\right)+C\right].
\end{align*}
We estimate the error term $E_2$ as follows:
\begin{align*}
|E_2|&\leq \left| \int_{-l}^l\left(\int_0^s\left[\TT'(\xi)-\TT'(0)\right]\xi\,{\rm d}\xi\right)\frac{{\rm d}s}{\left(|\hx|^2+s^2\right)^{3/2}}\right|+C\|\TT'\|^2_\infty\int_0^l\frac{s^3{\rm d}s}{\left(|\hx|^2+s^2\right)^{3/2}}\\
&\leq C' \left(|A(l,|\hx|)|+\|\TT'\|_{\infty}\right).
\end{align*}
Then, we write 
\[
I_2= -\TT(0)\wedge\TT'(0)\log{|\hx|}+E'_2,
\]
with
\[
|E'_2|\leq C\left(|A(l,|\hx|)|+\|\TT'\|_\infty\left[1+\log(\|\TT'\|_\infty)\right]\right).
\]
Summing up everything, we rewrite \eqref{main_part_BS} as follows:
\[
\int_{\RR}\frac{\TT(s)\wedge(\xx-\XX(s))}{|\xx-\XX(s)|^3}\,{\rm d}s=2\,\frac{\TT(0)\wedge\hx}{|\hx|^2}-\TT(0)\wedge\TT'(0)\log{|\hx|}+E_0+E_1+E'_2,
\]
where
\[
E_0=\int_{-l}^l\left[\TT(s)\wedge(\xx-\XX(s))\right]R(\hx,s)\,{\rm d}s+\int_{\RR\setminus[-l, l]}\frac{\TT(s)\wedge(\xx-\XX(s))}{|\xx-\XX(s)|^3}\,{\rm d}s.
\]
For the first term above, it follows from \eqref{equation noname 1} and \eqref{equation noname 2} that
\[
\left|\int_{-l}^l\left[\TT(s)\wedge(\xx-\XX(s))\right]R(\hx,s)\,{\rm d}s\right|\leq C\|\TT'\|_\infty\int_{-l}^l\frac{\left(|\hx|+\|\TT'\|_\infty s^2\right)\left(|\hx|+|s|+s^2\|\TT'\|_{\infty}\right)}{\left(|\hx|^2+s^2\right)^{3/2}}\,{\rm d}s,
\]
which can be estimated by the same arguments used before in this proof. For the second term in $E_0$, we use the uniform chord-arc condition \eqref{uniform_chord_arc_condition} to obtain
\[
|E_0|\leq C\|\TT'\|_{\infty}.
\]
This completes the proof.
\end{proof}
Now, we are in position to prove the main result of this section.
\begin{prop}\label{Lemma_v_o}
Let $\XX:\RR\to\RR^3$ be an open curve parametrized by arclength $s$, with $\XX\in C^{2,\alpha}$ for some $\alpha>0$. Assume that $\|\XX''\|_\infty=\|\TT'\|_{\infty}<\infty$. Assume also that the radius $R$ of the tubular neighborhood given by \eqref{definicion radio entorno tubular} is well defined and that there exists $c_0>0$ such that the chord-arc condition \eqref{uniform_chord_arc_condition_intro} holds. Let $\vo_0$ be as in \eqref{lamb_oseen_solution} and $K$ as in \eqref{definicion biot savart}. Then, for every $\xx\in\mathcal{T}_{R}$, the velocity induced by $\vo_0$ can be expressed in local coordinates $(r,\varphi,s)$ as
\begin{equation}
\begin{aligned}
	K*\vo_0:=\vv_0(r,\varphi,s,t)&=\frac{\Gamma}{2\pi}\frac{1}{r}\left(1-e^{-\frac{r^2}{4\nu t}}\right)\ea_\varphi(\varphi,s,t)-\frac{\Gamma}{4\pi}\log(\sqrt{\nu t})\TT(s,t)\wedge\TT_s(s,t)+\vv^*(r,\varphi,s,t)\\
	&=\vv_{L\text{-}O}+\XX_t+\vv^*,
\end{aligned}
\end{equation}
where, defining $A(l,r)$ as in \eqref{s_limit_condition}, $\vv^*$ satisfies
\begin{align}\label{forma explicita v estrella}
|\vv^*|\leq C \Gamma \|\TT'\|_\infty\left(\log{\left(1+\frac{r}{\sqrt{\nu t}}\right)}+\frac{|A(l,r)|}{\|\TT'\|_\infty}+\left|\log(\|\TT'\|_\infty)\right|+1+\left|\log\bigl(\sqrt{\nu t}\bigr)\right|e^{-\frac{R^2}{4\nu t}}\right)
\end{align}
for some constant $C>0$ independent of $\xx$ and $\XX$.
\end{prop}
\begin{proof}
Throughout this proof $C>0$ denotes a universal constant, independent of the geometric properties of the reference curve, which may change from line to line. As in Lemma \eqref{lem:BS-in-s}, we assume that $s_0=0$ and define $\XX_0=\XX(0)$ and $\TT_0=\TT(0)$. Due to the cutoff function $\eta_{R}$, the Biot--Savart integral is localized in the tubular neighborhood $\mathcal{T}_{R}$ and we can write
\[
\xx=\XX_0+\hx\qquad\mbox{and}\qquad \yy=\XX(s)+\hy(r,\varphi,s),
\]
where $\hy(r,\varphi,s)=r\ea_r(\varphi,s)$. A straightforward computation leads to
\begin{align*}
\vv_0(\xx,t)&=\frac{1}{4\pi}\int_{\RR^3}\vo_0(\yy,t)\wedge\frac{\xx-\yy}{|\xx-\yy|^3}\,{\rm d}\yy\\
&=\frac{1}{4\pi}\int_{\RR^2}\Omega^{\eta}(|\hy|,t)\left(\int_{\RR}\TT(s)\wedge\frac{\XX_0+\hx-\XX(s)-\hy(r,\varphi,s)}{|\XX_0+\hx-\XX(s)-\hy(r,\varphi,s)|^3}\,h(r,\varphi,s)\,{\rm d}s\right)\,{\rm d}\hy,
\end{align*}
where $h$ is the scale factor defined in \eqref{scale_factors}. Now, for fixed $(r,\varphi)$, consider the curve
\[
\XX_{\hy}(s)=\XX(s)+\hy(s)
\]
and a function $\varsigma=\varsigma(s)$ such that $\frac{{\rm d}\varsigma}{{\rm d}s}=h(s)$. Then, $\tilde{\XX}(\varsigma)=\XX_{\hy}(s(\varsigma))$ is parametrized by arclength $\varsigma$. Renaming $\tilde{\XX}_0=\XX_0+\hy(0)$ and $\tilde{\bm{x}}_{\perp}=\hx-\hy(0)$, we get
\begin{align}\label{integral cambio de variables}
\vv_0(\xx,t)=\frac{1}{4\pi}\int_{\RR^2}\Omega^{\eta}(|\hy|,t)\left(\int_{\RR}\tilde{\TT}(\varsigma)\wedge\frac{\tilde{\XX}_0+\tilde{\bm{x}}_{\perp}-\tilde{\XX}(\varsigma)}{|\tilde{\XX}_0+\tilde{\bm{x}}_{\perp}-\tilde{\XX}(\varsigma)|^3}\,{\rm d}\varsigma\right)\,{\rm d}\hy.
\end{align}
Hence, by Lemma \ref{lem:BS-in-s}, it follows that
\[
\int_{\RR}\tilde{\TT}(\varsigma)\wedge\frac{\tilde{\bm{x}}_{\perp}-(\tilde{\XX}(z)-\tilde{\XX}_0)}{|\tilde{\bm{x}}_{\perp}-(\tilde{\XX}(\varsigma)-\tilde{\XX}_0)|^3}\,{\rm d}\varsigma
=
2\frac{\TT_0\wedge\tilde{\bm{x}}_{\perp}}{|\tilde{\bm{x}}_{\perp}|^2}
-\log|\tilde{\bm{x}}_{\perp}|\,(\TT_0\wedge\TT'_0)
+\tilde{\vv}^*.
\]
Therefore,
\begin{align*}
\vv_0(\xx,t)
=&\frac{1}{2\pi}\int_{\RR^2}\Omega^{\eta}(|\hy|,t)\frac{\TT_0\wedge(\hx-\hy(0))}{|\hx-\hy(0)|^2}\,{\rm d}\hy\\
&-\frac{\TT_0\wedge\TT'_0}{4\pi}\int_{\RR^2}\Omega^{\eta}(|\hy|,t)\log|\hx-\hy(0)|\,{\rm d}\hy\\
&+\frac{1}{4\pi}\int_{\RR^2}\Omega^{\eta}(|\hy|,t)\tilde{\vv}^*\,{\rm d}\hy\\
:=&\,I_1+I_2+I_3.
\end{align*}
Now, we estimate $I_1$, $I_2$ and $I_3$ independently. In $I_1$ we identify the velocity generated by the Lamb--Oseen vortex, an infinite straight vortex filament of strength $\Gamma$. In fact, after the change of variables $\rho=r/\sqrt{\nu t}$, and writing $\rho_0=|\hx|/\sqrt{\nu t}$, we obtain
\begin{align}\label{integral_lamb_oseen}
I_1=\frac{\Gamma}{8\pi^2}\frac{1}{\sqrt{\nu t}}\TT_0\wedge\frac{\hx}{|\hx|}\int_0^{\infty}\eta_R(\sqrt{\nu t}\rho)e^{-\frac{\rho^2}{4}}\left(\int_0^{2\pi}\frac{\rho_0-\rho e^{i\varphi}}{|\rho_0-\rho e^{i\varphi}|^2}\,{\rm d}\varphi\right)\rho\,{\rm d}\rho.
\end{align}
Notice that the sine component vanishes by symmetry, so the angular integral reduces to
\[
\int_0^{2\pi}\frac{\rho_0-\rho e^{i\varphi}}{|\rho_0-\rho e^{i\varphi}|^2}\,{\rm d}\varphi
=
\int_0^{2\pi}
\frac{\rho_0-\rho\cos\varphi}{\rho_0^2+\rho^2-2\rho_0\rho\cos\varphi}
\,{\rm d}\varphi.
\]
Using the following identity of logarithmic kernels
\begin{align}\label{log_identity}
\int_0^{2\pi}\log\bigl|\rho_0-\rho e^{i\varphi}\bigr|\,{\rm d}\varphi
=
2\pi \log\bigl(\max\{\rho_0,\rho\}\bigr),
\end{align}
differentiation with respect to $\rho_0$ yields
\[
\int_0^{2\pi}
\frac{\rho_0-\rho\cos\varphi}{\rho_0^2+\rho^2-2\rho_0\rho\cos\varphi}
\,{\rm d}\varphi
=
\begin{cases}
\dfrac{2\pi}{\rho_0}, & \rho<\rho_0,\\[1ex]
0, & \rho>\rho_0.
\end{cases}
\]
Then, back to \eqref{integral_lamb_oseen}, we see that $I_1$ corresponds to Lamb-Oseen velocity term. In fact,
\[
I_1=\frac{\Gamma}{4\pi}\frac{1}{\sqrt{\nu t}}\frac{\ea_{\varphi_0}}{\rho_0}
\int_0^{\rho_0}e^{-\frac{\rho^2}{4}}\rho\,{\rm d}\rho
=
\frac{\Gamma}{2\pi\sqrt{\nu t}}\,
\frac{\left(1-e^{-\rho_0^2/4}\right)}{\rho_0}\,
\ea_{\varphi_0}.
\]
We continue with $I_2$, in which we identify the binormal-flow velocity term $\XX_t$ defined in \eqref{definicion binormal flow}:
\begin{align*}
I_2=&-\frac{\Gamma}{4\pi}\frac{\TT_0\wedge\TT'_0}{2}\log{\sqrt{\nu t}}\int_0^{\infty}\eta_R(\sqrt{\nu t}\rho)e^{-\frac{\rho^2}{4}}\rho\,{\rm d}\rho\\
&-\frac{\Gamma}{4\pi}\frac{\TT_0\wedge\TT'_0}{4\pi}\int_0^{\infty}\eta_R(\sqrt{\nu t}\rho)e^{-\frac{\rho^2}{4}}\left(\int_0^{2\pi}\log|\rho_0-\rho e^{i\varphi}|\,{\rm d}\varphi\right)\rho\,{\rm d}\rho\\
=&\,\XX_t(0)+E_2.
\end{align*}
where
\begin{align*}
E_2=&\frac{\Gamma}{4\pi}\frac{\TT_0\wedge\TT'_0}{2}\log{\sqrt{\nu t}}\int_0^{\infty}\left(1-\eta_R(\sqrt{\nu t}\rho)\right)e^{-\frac{\rho^2}{4}}\rho\,{\rm d}\rho\\
&-\frac{\Gamma}{4\pi}\frac{\TT_0\wedge\TT'_0}{4\pi}\int_0^{\infty}\eta_R(\sqrt{\nu t}\rho)e^{-\frac{\rho^2}{4}}\left(\int_0^{2\pi}\log|\rho_0-\rho e^{i\varphi}|\,{\rm d}\varphi\right)\rho\,{\rm d}\rho\\
=&\,E_{2,1}+E_{2,2}.
\end{align*}
It is straightforward that
\[
|E_{2,1}|\leq C\Gamma\|\TT'\|_{\infty}\left|\log\bigl(\sqrt{\nu t}\bigr)\right|e^{-\frac{R^2}{4\nu t}},
\]
Using \eqref{log_identity} and splitting the radial integral at $\rho=\rho_0$, we obtain
\[
|E_{2,2}|\leq C\Gamma\|\TT'\|_{\infty}\left(|\log\rho_0|\int_0^{\rho_0}e^{-\frac{\rho^2}{4}}\rho\,{\rm d}\rho+\int_{\rho_0}^{\infty}e^{-\frac{\rho^2}{4}}\rho|\log\rho|\,{\rm d}\rho\right).
\]
Then, if $0<\rho_0\leq1$, then
\[
|\log\rho_0|\int_0^{\rho_0}e^{-\frac{\rho^2}{4}}\rho\,{\rm d}\rho\leq \frac{\rho_0^2}{2}|\log\rho_0|<\infty,
\]
and
\[
\int_{\rho_0}^{\infty}e^{-\frac{\rho^2}{4}}\rho|\log\rho|\,{\rm d}\rho\leq \int_0^{\infty}e^{-\frac{\rho^2}{4}}\rho|\log\rho|\,{\rm d}\rho< \infty.
\]
On the other hand, if $\rho_0\geq1$, then
\[
|\log\rho_0|\int_0^{\rho_0}e^{-\frac{\rho^2}{4}}\rho\,{\rm d}\rho\leq 2\log\rho_0,
\]
while
\[
\int_{\rho_0}^{\infty}e^{-\frac{\rho^2}{4}}\rho\log\rho\,{\rm d}\rho\leq \int_1^{\infty}e^{-\frac{\rho^2}{4}}\rho\log\rho\,{\rm d}\rho< \infty.
\]
Therefore, in both cases,
\[
|E_{2,2}|\leq C\Gamma\|\TT'\|_{\infty}\bigl(1+\log(1+\rho_0)\bigr).
\]
Finally, the term $I_3$ is estimated directly by using the estimate of $\tilde{\vv}^*$ in Lemma \ref{lem:BS-in-s}, giving
\[
|I_3|\leq C\Gamma\left(|A(l,|\hx|)|+\|\TT'\|_\infty\left[1+\left|\log(\|\TT'\|_\infty)\right|\right]\right).
\]
Collecting the expressions for $I_1$, $I_2$ and $I_3$, and absorbing the error terms into $\vv^*$, we obtain
\[
\vv_0=\vv_{L\text{-}O}+\XX_t+\vv^*,
\]
with $\vv^*$ satisfying \eqref{forma explicita v estrella}. This completes the proof.
\end{proof}

\section{Navier--Stokes in local coordinates}\label{sec: Navier Stokes curvilinear}
We now write the vorticity formulation of the Navier--Stokes equations in the local coordinates constructed in Section \ref{sec: the local frame}. The purpose of this section is to compute the equation satisfied by the perturbation around the leading Lamb--Oseen profile localized around the reference curve. More precisely, we look for a solution of \eqref{vorticity_NS} in the form
\begin{equation}\label{full_solution_NS}
\vo=\vo_0+\tilde{\vo},
\end{equation}
where $\vo_0$ is the leading order profile defined in \eqref{lamb_oseen_solution}. 

As shown in Proposition \ref{Lemma_v_o}, the associated approximate velocity field is given by
\begin{equation}\label{velocity_decomposition_NS}
\begin{aligned}
\vv_0(r,\varphi,s,t)=&\frac{\Gamma}{2\pi}\frac{1}{r}\left(1-e^{-\frac{r^2}{4\nu t}}\right)\ea_\varphi(\varphi,s,t)-\frac{\Gamma}{4\pi}\log(\sqrt{\nu t})\TT(s,t)\wedge\TT_s(s,t)+\vv^*(r,\varphi,s,t)\\
&=\vv_{L\text{-}O}+\XX_t+\vv^*.
\end{aligned}
\end{equation}
Here $\vv_{L\text{-}O}$ denotes the Lamb--Oseen velocity, $\XX_t$ is the velocity of the reference curve under the binormal flow and $\vv^*$ is the nonlocal remainder obtained from the Biot--Savart law. Then, the velocity induced by \eqref{full_solution_NS} is decomposed as
\begin{equation}\label{full_velocity_NS}
\vv=K*\vo=\vv_0+\tilde{\vv},\quad\mbox{with }\quad \tilde{\vv}=K*\tilde{\vo}.
\end{equation}

Following the strategy of \cite{FontelosVega}, we now substitute the decomposition \eqref{full_solution_NS}--\eqref{full_velocity_NS} into \eqref{vorticity_NS}. This gives an equation for the perturbation $\tilde{\vo}$, whose inhomogeneous term is the defect produced by the approximate pair $(\vo_0,\vv_0)$. More precisely, the perturbation satisfies
\begin{equation}\label{NS_omega_tilde}
\partial_t\tilde{\vo}-\nu\Delta\tilde{\vo}=NL[\tilde{\vo},\tilde{\vv}]+L[\vo_0,\tilde{\vo}]+\bm{F}[\vo_0,\vv_0].
\end{equation}
Here the nonlinear and linear terms are defined by
\begin{align}\label{nonlinear_term_NS}
NL[\tilde{\vo},\tilde{\vv}]&=-(\tilde{\vv}\cdot\nabla)\tilde{\vo}+(\tilde{\vo}\cdot\nabla)\tilde{\vv},\\
L[\vo_0,\tilde{\vo}]&=-(\tilde{\vv}\cdot\nabla)\vo_0-(\vv_0\cdot\nabla)\tilde{\vo}+(\tilde{\vo}\cdot\nabla)\vv_0+(\vo_0\cdot\nabla)\tilde{\vv}\label{linear_term_NS},
\end{align}
and the forcing term is the defect of the approximate solution, namely
\begin{equation}\label{forcing_defect}
\bm{F}[\vo_0,\vv_0]=-\partial_t\vo_0+\nu\Delta\vo_0-(\vv_0\cdot\nabla)\vo_0+(\vo_0\cdot\nabla)\vv_0.
\end{equation}
\begin{lem}\label{ecuacion termino de fuerza 3/2}
Let $\vo_0$ and $\vv_0$ be given by \eqref{lamb_oseen_solution} and \eqref{velocity_decomposition_NS}, respectively. Then $\bm{F}[\vo_0,\vv_0]$ in \eqref{forcing_defect} is of order $(\nu t)^{-3/2}$, for $(\nu t)\ll1$, and is given explicitly by
\begin{equation}\label{termino de fuerza 1}
\begin{aligned}
	\bm{F}[\vo_0,\vv_0]=&-\Omega^\eta s_t\TT_s-\Omega^\eta\TT_t-\vv^*\cdot\nabla\vo_0+\vo_0\cdot\nabla\vv_{L\text{-}O}+\frac{1}{h}\Omega^\eta\partial_s\XX_t+\vo_0\cdot\nabla\vv^*\\
	&+\nu\left(\frac{1}{r}\Omega\eta_{R,r}+2\Omega_r\eta_{R,r}+\Omega\eta_{R,rr}-\frac{\ea_r\cdot\TT_s}{h}\Omega\eta_{R,r}\right)\TT\\
	&+\nu\left(-\frac{\ea_r\cdot\TT_s}{h}\Omega_r\eta_R\TT+\Omega^\eta\Delta\TT\right).
\end{aligned}
\end{equation}
\end{lem}
\begin{proof}
We proceed to compute all terms in
\begin{equation}\label{NS_omega_0}
\bm{F}[\vo_0,\vv_0]=-\partial_t\vo_0+\nu\Delta\vo_0-(\vv_0\cdot\nabla)\vo_0+(\vo_0\cdot\nabla)\vv_0.
\end{equation}
We use the notation introduced in Section \ref{sec: biot savart integral}, so that
\[
\vo_0(r,s,t)=\Omega^\eta(r,t)\TT(s,t),
\qquad
\Omega^\eta(r,t)=\Omega(r,t)\eta_R(r).
\]
In the computations below, when no confusion can arise, we omit the explicit dependence of $\Omega$ and $\Omega^\eta$ on $t$. Here and below, when applying spatial derivatives to $\TT$, we regard $\TT(s,t)$ as the vector field $\TT(s(\xx,t),t)$ in the tubular neighborhood, thus $\Delta\TT$ denotes the spatial Laplacian of this extension. Then, we have that
\[
\frac{\partial \vo_0}{\partial x_i}=\Omega^\eta_i\TT+\Omega^\eta\TT_i,\qquad\mbox{where}\qquad\Omega^\eta_i=\Omega^\eta_r r_i.
\]
We obtain the following expression for the Laplacian of $\vo_0$:
\[
\Delta\vo_0=\sum_{i=1}^3\left(\Omega^\eta_{rr}r_i^2\TT+\Omega_r^\eta r_{ii} \TT+2\Omega^\eta_i\TT_i+\Omega^\eta\TT_{ii}\right).
\]
Using \eqref{eq:derivadas_coordenadas}, one has
\[
\sum_{i=1}^3r_i^2=1,\qquad\sum_{i=1}^3\Omega^\eta_i\TT_i=\Omega^\eta_r\TT_s\sum_{i=1}^3r_is_i=0\qquad\mbox{and}\qquad\sum_{i=1}^3r_{ii}=\frac{1}{r}-\frac{\ea_r\cdot\TT_s}{h},
\]
and therefore
\begin{align*}
\Delta\vo_0&=\Omega^\eta_{rr}\TT+\frac{1}{r}\Omega^\eta_{r}\TT-\frac{\ea_r\cdot\TT_s}{h}\Omega^\eta_{r}\TT+\Omega^\eta\Delta\TT\\
&=\left(\Omega_{rr}+\frac{1}{r}\Omega_r\right)\eta_R\TT+\left(\frac{1}{r}\Omega\eta_{R,r}+2\Omega_r\eta_{R,r}+\Omega\eta_{R,rr}-\frac{\ea_r\cdot\TT_s}{h}\Omega\eta_{R,r}\right)\TT\\
&\quad-\frac{\ea_r\cdot\TT_s}{h}\Omega_r\eta_R\TT+\Omega^\eta\Delta\TT.
\end{align*}
Similar arguments show that $\vv_{L\text{-}O}\cdot\nabla\vo_0=0$, so the transport term takes the form
\[
\vv_0\cdot\nabla\vo_0=\XX_t\cdot\nabla\vo_0+\vv^*\cdot\nabla\vo_0.
\]
Furthermore, the stretching term yields
\[
\vo_0\cdot\nabla\vv_0=\vo_0\cdot\nabla\vv_{L\text{-}O}+\frac{1}{h}\Omega^\eta\partial_s\XX_t+\vo_0\cdot\nabla\vv^*.
\]
Since
\[
\vo_0(\xx,t)=\Omega^\eta(r(\xx,t),t)\TT(s(\xx,t),t),
\]
when differentiating with respect to time, the Cartesian point $\xx$ is kept fixed. Thus the local coordinates $r=r(\xx,t)$, $\varphi=\varphi(\xx,t)$ and $s=s(\xx,t)$ also depend on time. In the following formula, $\partial_t\Omega^\eta$ denotes the explicit time derivative at fixed $r$, while $\TT_t$ denotes the time derivative of $\TT(s,t)$ at fixed $s$. Therefore,
\[
\partial_t\vo_0
=
\partial_t\Omega^\eta\TT
+\Omega^\eta_r r_t\TT
+\Omega^\eta s_t\TT_s
+\Omega^\eta\TT_t.
\]
A straightforward computation shows that the leading terms of order $(\nu t)^{-2}$ without derivatives of the cutoff are
\[
\nu\left(\Omega_{rr}+\frac{1}{r}\Omega_r\right)\eta_R\TT
\qquad\mbox{and}\qquad
\partial_t\Omega\,\eta_R\TT.
\]
Since the Lamb--Oseen profile satisfies the two-dimensional heat equation,
\[
\partial_t\Omega=\nu\left(\Omega_{rr}+\frac{1}{r}\Omega_r\right),
\]
these two terms cancel out in \eqref{NS_omega_0}. The remaining terms involving derivatives of the cutoff are precisely those displayed in \eqref{termino de fuerza 1}.
Then, the terms of order $\log(\nu t)(\nu t)^{-3/2}$, namely
\[
-\XX_t\cdot\nabla\vo_0
\qquad\mbox{and}\qquad
-\Omega^\eta_r r_t\TT,
\]
also cancel out. Indeed, since $\XX_t\cdot\TT=0$, we have
\[
\XX_t\cdot\nabla\vo_0
=
\Omega^\eta_r(\XX_t\cdot\ea_r)\TT
+\Omega^\eta(\XX_t\cdot\nabla)\TT.
\]
The first term cancels with $-\Omega^\eta_r r_t\TT$, because $r_t=-\XX_t\cdot\ea_r$ by \eqref{eq:derivadas_tiempo_coordenadas}. The remaining term $\Omega^\eta(\XX_t\cdot\nabla)\TT$ is part of the lower-order contribution. This cancellation reflects the fact that the tubular coordinates are transported by the binormal motion of the reference curve. Therefore, after these two cancellations, the remaining terms are of order at most $(\nu t)^{-3/2}$.

Finally, collecting the remaining terms in \eqref{NS_omega_0}, we obtain
\begin{align*}
\bm{F}[\vo_0,\vv_0]=&-\Omega^\eta s_t\TT_s-\Omega^\eta\TT_t-\vv^*\cdot\nabla\vo_0+\vo_0\cdot\nabla\vv_{L\text{-}O}+\frac{1}{h}\Omega^\eta\partial_s\XX_t+\vo_0\cdot\nabla\vv^*\\
&+\nu\left(\frac{1}{r}\Omega\eta_{R,r}+2\Omega_r\eta_{R,r}+\Omega\eta_{R,rr}-\frac{\ea_r\cdot\TT_s}{h}\Omega\eta_{R,r}\right)\TT\\
&+\nu\left(-\frac{\ea_r\cdot\TT_s}{h}\Omega_r\eta_R\TT+\Omega^\eta\Delta\TT\right),
\end{align*}
which is precisely \eqref{termino de fuerza 1}.
\end{proof}
\begin{rem}\label{rem: asymptotics fontelos vega}
As stated in Lemma \ref{ecuacion termino de fuerza 3/2}, the forcing term $\bm{F}[\vo_0,\vv_0]$ is of order $O((\nu t)^{-3/2})$. Its leading order contribution is given by
\[
\bm{f}=-\nu\frac{\ea_r\cdot\TT_s}{h}\Omega_r\eta_R\TT-\vv^*\cdot\nabla\vo_0+\vo_0\cdot\nabla\vv_{L\text{-}O}.
\]
The first term is of order $O(\nu\Gamma)$, whereas the remaining terms are of order $O(\Gamma^2)$. Since in Section \ref{sec: existence result} we shall work in the regime $\Gamma/\nu$ small with $\nu$ of order $O(1)$, the dominant contribution is
\[
\bm{f}_1=-\nu\frac{\ea_r\cdot\TT_s}{h}\Omega_r\eta_R\TT=-\nu\frac{\Gamma}{8\pi}\kappa\rho\cos(\varphi-\theta)\frac{e^{-\rho^2/4}}{(\nu t)^{3/2}}\eta_R\TT.
\]
This term does not obstruct the construction of a solution $\tilde{\vo}$ in Section \ref{sec: existence result}. Nevertheless, at first order, the perturbation is driven by the Duhamel contribution associated with $\bm{f}_1$. For this reason, the authors of \cite{FontelosVega} introduce a correction to the leading vorticity profile. This correction is obtained from the model problem in which, at each point of the reference curve, the arclength variable $s$ is treated as a parameter and the reference curve is replaced by its tangent line. In this way one solves a two-dimensional heat equation in the normal variables,
\begin{equation}\label{eq:omega_correction}
\partial_t\vo_1-\nu\Delta_{\hx}\vo_1=\bm{f}_1,
\end{equation}
where $\Delta_{\hx}$ denotes the Laplacian in the plane perpendicular to the reference curve. The corresponding first-order correction is given by
\[
\vo_1(\xx,t)=\frac{\Gamma}{8\pi}\kappa\rho\frac{e^{-\frac{\rho^2}{4}}}{\sqrt{\nu t}}\cos(\theta(s))\TT(s),
\]
see Theorem \ref{thm: fontelos vega}. We will not need this correction in the existence argument below, but it is useful to note that the dominant term $\bm f_1$ is formally the first-order contribution captured by the correction introduced in \cite{FontelosVega}.
\end{rem}
\begin{rem}\label{rem: cancelacion torsion}
It will be essential later in Section \ref{sec: moving filament} to understand how the forcing term in \eqref{NS_omega_tilde} depends on the geometric quantities of the reference curve, and in particular on the torsion. There are three possible sources of such a dependence. First, the terms involving the Biot--Savart remainder $\vv^*$ inherit their geometric dependence from the estimate in Lemma \ref{lem:BS-in-s}, and in particular from the quantity $A(l,|\hx|)$. This dependence will be analysed for the Hasimoto soliton in Section \ref{sec: moving filament}. Second, the terms involving second $s$-derivatives of $\TT$ that are written in divergence form; these terms can be handled in the Duhamel formula by transferring the derivatives onto the heat kernel. Finally, there are terms involving second $s$-derivatives of $\TT$ which are not in divergence form. These are the terms that we isolate here:
\[
\Omega^\eta s_t\TT_s+\Omega^\eta\TT_t-\frac{1}{h}\Omega^\eta\partial_s\XX_t.
\]
Using the properties of the local frame established in Section \ref{sec: the local frame}, the time derivative of the arclength coordinate in \eqref{eq:derivadas_tiempo_coordenadas} can be written as
\[
s_t=-\frac{r}{h}\mathsf{c}\,\ea_{\varphi}\cdot\TT_{ss}.
\]
Moreover, since $\TT_t=\partial_s\XX_t$, we obtain
\[
\Omega^\eta s_t\TT_s+\Omega^\eta\TT_t-\frac{1}{h}\Omega^\eta\partial_s\XX_t
=
-\frac{r}{h}\Omega^\eta\mathsf{c}\left[(\ea_{\varphi}\cdot\TT_{ss})\TT_s+(\ea_r\cdot\TT_s)\TT\wedge\TT_{ss}\right].
\]
After rewriting this expression in the parallel frame, the dependence on the second derivative $\TT_{ss}$ reduces to a derivative of the curvature squared. More precisely,
\begin{equation}\label{segunda cancelación temporal}
\Omega^\eta s_t\TT_s+\Omega^\eta\TT_t-\frac{1}{h}\Omega^\eta\partial_s\XX_t
=
-\frac{r}{h}\Omega^\eta\frac{\Gamma}{4\pi}\log(\sqrt{\nu t})\frac{1}{2}\partial_s\left(|\TT_s|^2\right)\ea_{\varphi}.
\end{equation}
In particular, the non-divergence terms above do not contain any explicit dependence on the torsion. Although this observation is not needed in the existence argument below, it will be relevant in Section \ref{sec: moving filament}. With this simplification, the forcing term may be written as
\begin{equation*}
\begin{aligned}
	\bm{F}[\vo_0,\vv_0]=&\frac{r}{h}\frac{\Gamma}{4\pi}\log(\sqrt{\nu t})\Omega^\eta\frac{1}{2}\partial_s\left(|\TT_s|^2\right)\ea_{\varphi}-\vv^*\cdot\nabla\vo_0+\vo_0\cdot\nabla\vv_{L\text{-}O}+\vo_0\cdot\nabla\vv^*\\
	&+\nu\left(\frac{1}{r}\Omega\eta_{R,r}+2\Omega_r\eta_{R,r}+\Omega\eta_{R,rr}-\frac{\ea_r\cdot\TT_s}{h}\Omega\eta_{R,r}\right)\TT\\
	&+\nu\left(-\frac{\ea_r\cdot\TT_s}{h}\Omega_r\eta_R\TT+\Omega^\eta\Delta\TT\right).
\end{aligned}
\end{equation*}
\end{rem}
\section{Existence result}\label{sec: existence result}
In this section we prove that the perturbation equation obtained in \eqref{NS_omega_tilde} admits a solution in a suitable Morrey-type space. The argument adapts the strategy of \cite{FontelosVega} to the functional setting of \cite{GigaMiyakawa}. More precisely, we consider
\begin{equation}\label{NS_tilde_final}
\begin{cases}
\partial_t\tilde{\vo}-\nu\Delta\tilde{\vo}=NL[\tilde{\vo},\tilde{\vv}]+L[\vo_0,\tilde{\vo}]+\bm{F}[\vo_0,\vv_0],&\text{in }\RR^3\times(0,T),\\
\tilde{\vv}=K*\tilde{\vo},\qquad \tilde{\vo}|_{t=0}=0,&\text{in }\RR^3,
\end{cases}
\end{equation}
where $NL[\tilde{\vo},\tilde{\vv}]$, $L[\vo_0,\tilde{\vo}]$, and $\bm{F}[\vo_0,\vv_0]$ are defined in \eqref{nonlinear_term_NS}, \eqref{linear_term_NS}, and \eqref{termino de fuerza 1}, respectively. By Duhamel's formula, any solution of \eqref{NS_tilde_final} satisfies
\[
\tilde{\vo}(t)=\tilde{\vo}_F(t)+\int_0^t e^{\nu(t-\tau)\Delta}\left(NL[\tilde{\vo},\tilde{\vv}]+L[\vo_0,\tilde{\vo}]\right)(\tau)\,{\rm d}\tau,
\]
where the contribution generated by the forcing term is
\begin{equation}\label{definicion omega_F}
\tilde{\vo}_F(t):=\int_0^t e^{\nu(t-\tau)\Delta}\bm{F}[\vo_0,\vv_0](\tau)\,{\rm d}\tau.
\end{equation}
We shall first estimate $\tilde{\vo}_F$ and then prove that the remaining Duhamel terms define a contraction for sufficiently small $\Gamma/\nu$ and $\nu T$.

\subsection{Functional setting}\label{subsec: functional setting}
We now recall the functional framework that will be used in the fixed point argument. Following \cite{GigaMiyakawa}, we work with Morrey-type norms for Radon measures. These norms are well suited to vorticities concentrated around lower-dimensional sets, while for positive times they can also be applied to the regularized densities produced by the heat flow.

\begin{defn}\label{Morrey_giga}
Let $\mu$ be a Radon measure on $\RR^3$ and let $1\leq p\leq \infty$. We define
\[
\|\mu\|_p:=\sup_{x\in\RR^3,\ r>0}r^{-\frac{3}{p'}}|\mu|(B(x,r)),
\]
with $p'=\frac{p}{p-1}$. The Morrey-$p$ space is
\[
\mathcal{M}^p:=\left\{\mu\in\mathcal{M}_{\mathrm{Radon}}(\RR^3):\|\mu\|_p<\infty\right\}.
\]
Here $|\mu|(B(x,r))$ denotes the total variation of $\mu$ in the ball of radius $r$ centered at $x$. If $d\mu=f\,dx$, we write $\|f\|_p$ for the Morrey norm of the measure $\mu$.
\end{defn}

We next collect the estimates in Morrey spaces given in \cite{GigaMiyakawa} that will be used throughout the fixed point argument. The first ones concern the heat semigroup and provide the time singularities that appear in the Duhamel formula. The second ones concern the Biot--Savart operator and will be used to estimate the velocity induced by a vorticity in Morrey spaces.

\begin{prop}[\cite{GigaMiyakawa}]\label{prop:Morrey_basic}
Let $\mathcal{M}^p$ be defined as in Definition \ref{Morrey_giga}. Then:
\begin{enumerate}[label=(\roman*)]
\item $\mathcal{M}^p$ is a Banach space with norm $\|\cdot\|_p$.
\item $\mathcal{M}^1$ is the space of finite measures, and $\|\mu\|_1=|\mu|(\RR^3)$.
\item $\mathcal{M}^{\infty}$ coincides with $L^\infty$ after identifying an $L^\infty$ function with the measure $f\,dx$, and the norms are equivalent.
\end{enumerate}
\end{prop}

For $T>0$, we define the time-weighted Morrey norm
\begin{equation}\label{def: temporal Morrey norm}
\|f\|_{T,p}:=\sup_{0<t<T}(\nu t)^{\frac12-\frac{3}{2p}}\|f(t)\|_p.
\end{equation}
We denote by $X_{T,p}$ the space of all functions $f:(0,T)\to\mathcal{M}^p$ such that $\|f\|_{T,p}<\infty$. Since $\mathcal{M}^p$ is a Banach space, $X_{T,p}$, endowed with the norm \eqref{def: temporal Morrey norm}, is a Banach space. This is the time-weighted Morrey space used in the fixed point argument, following the framework of \cite{GigaMiyakawa}.

The heat semigroup satisfies the following estimates in Morrey spaces.

\begin{prop}[\cite{GigaMiyakawa}]\label{proposition Giga Miya convolucion calor}
Let $1\leq q\leq p\leq \infty$. Then there exists a constant $C>0$, independent of $\mu$ and $t>0$, such that
\begin{enumerate}[label=(\roman*)]
\item $\|e^{\nu t\Delta}\mu\|_q\leq \|\mu\|_q$.
\item $\|\nabla e^{\nu t\Delta}\mu\|_p\leq C(\nu t)^{-\frac12-\frac12(\frac3q-\frac3p)}\|\mu\|_q$.
\end{enumerate}
\end{prop}

We shall also use the following bounds for the Biot--Savart operator.

\begin{prop}[\cite{GigaMiyakawa}]\label{proposition Giga Miya convolucion biot savart}
Let $K$ be the Biot--Savart operator.
\begin{enumerate}[label=(\roman*)]
\item If $\frac1p=\frac13+\frac1q$ and $\mu\in\mathcal{M}^p$, then $K*\mu\in\mathcal{M}^q\cap L^1_{\mathrm{loc}}$, and
\[
\|K*\mu\|_q\leq C\|\mu\|_p.
\]
\item If $q<3<p$ and $\mu\in\mathcal{M}^p\cap\mathcal{M}^q$, then $K*\mu\in L^\infty$, and
\[
\|K*\mu\|_\infty\leq C\|\mu\|_p^{\frac{2/3-1/q'}{1/q-1/p}}\|\mu\|_q^{\frac{1/p'-2/3}{1/q-1/p}}.
\]
\end{enumerate}
In both estimates, $C$ is independent of $\mu$.
\end{prop}
\subsection{Estimates for the forcing term}
We now estimate the contribution generated by the forcing term in \eqref{definicion omega_F}. The estimates rely on the fact that the forcing is localized in a tubular neighborhood of the reference curve and has Gaussian decay in the normal direction. We first record a simple Morrey estimate for such Gaussian profiles.

Throughout this section, we use the uniform chord--arc condition \eqref{uniform_chord_arc_condition}. In particular, this condition prevents long portions of the reference curve from concentrating inside a small ball.

\begin{lem}\label{lema 5.5}
Let $3/2\leq q\leq\infty$, and let $\eta_R$ be as defined in \eqref{definicion cutoff}. Then, for any $c>0$, there exists $C>0$ such that
\begin{equation}\label{tubular_gaussian_morrey_estimate}
\left\|e^{-c\frac{r^2}{\nu t}}\eta_R(r)\right\|_q\leq C(\nu t)^{\frac{3}{2q}},
\end{equation}
for every $t>0$.
\end{lem}

\begin{proof}
We estimate the Morrey norm of $e^{-c\frac{r^2}{\nu t}}\eta_R(r)$. Let $x\in\RR^3$ and $r_0>0$. We have to bound
\[
\frac{1}{r_0^{\frac{3}{q'}}}\int_{B(x,r_0)}e^{-c\frac{\operatorname{dist}(\yy,\XX)^2}{\nu t}}\eta_R(\operatorname{dist}(\yy,\XX))\,{\rm d}\yy.
\]
We distinguish two regimes, according to whether $r_0\leq\sqrt{\nu t}$ or $r_0>\sqrt{\nu t}$.

If $r_0\leq\sqrt{\nu t}$, then
\[
\frac{1}{r_0^{\frac{3}{q'}}}\int_{B(x,r_0)}e^{-c\frac{\operatorname{dist}(\yy,\XX)^2}{\nu t}}\eta_R(\operatorname{dist}(\yy,\XX))\,{\rm d}\yy\leq C r_0^{3-\frac{3}{q'}}\leq C(\nu t)^{\frac{3}{2q}}.
\]

Assume now that $r_0>\sqrt{\nu t}$. We define
\[
I(x,r_0,\XX):=\left\{s\in\RR:\exists\,\hy\in\TT^\perp(s,t)\ \text{such that}\ |\hy|\leq R\ \text{and}\ \XX(s,t)+\hy\in B(x,r_0)\right\},
\]
where $\TT^\perp(s,t)$ denotes the plane orthogonal to $\TT(s,t)$. By the non-degeneracy of the tubular coordinates and the chord-arc condition \eqref{chord_arc_condition}, we have
\[
|I(x,r_0,\XX)|\leq C r_0.
\]
Using local coordinates in the tubular neighborhood, we obtain
\begin{align*}
\frac{1}{r_0^{\frac{3}{q'}}}\int_{B(x,r_0)}e^{-c\frac{\operatorname{dist}(\yy,\XX)^2}{\nu t}}\eta_R(\operatorname{dist}(\yy,\XX))\,{\rm d}\yy
&\leq\frac{C}{r_0^{\frac{3}{q'}}}|I(x,r_0,\XX)|\int_0^R e^{-c\frac{r^2}{\nu t}}r\,{\rm d}r\\
&\leq C r_0^{1-\frac{3}{q'}}(\nu t)\\
&\leq C(\nu t)^{\frac{3}{2q}}.
\end{align*}
In the last inequality we have used $1-\frac{3}{q'}\leq0$, which follows from $q\geq3/2$, together with $r_0>\sqrt{\nu t}$. Taking the supremum over $x\in\RR^3$ and $r_0>0$, we obtain \eqref{tubular_gaussian_morrey_estimate}.
\end{proof}
We now estimate the Duhamel contribution generated by the forcing term. The estimate is based on the decomposition of the forcing obtained in Remark \ref{rem: cancelacion torsion} and on the estimates given in Section \ref{subsec: functional setting}.

\begin{lem}\label{estimacion 3/2 termino de fuerza}
Let $\bm{F}[\vo_0,\vv_0]$ be defined as in \eqref{termino de fuerza 1}, and assume that $\XX(\cdot,t)\in C^3$ for all $t\in(0,T)$. Suppose moreover that $\sqrt{\nu T}\ll1$ and that $\Gamma/\nu=O(1)$. Let
\[
l=\frac{1}{5\|\TT_s\|_{\infty}},
\]
and let $A(l,|\hx|)$ be the quantity defined in \eqref{s_limit_condition}. We define
\begin{equation}\label{constante explicita C_F}
\mathcal{C}_F(\XX,\Gamma,\nu T):=\Gamma\left(\|\TT_s\|_{\infty}+\frac{\Gamma}{\nu}\left[\|A(l,|\hx|)\|_{\infty}+\|\TT_s\|_\infty\left|\log(\|\TT'\|_\infty)\right|\right]+O\left(|\log(\nu T)|\sqrt{\nu T}\right)\right).
\end{equation}
Then, for every $3/2\leq p\leq\infty$,
\begin{equation}\label{cota termino de fuerza}
\|\tilde{\vo}_F\|_{T,p}\leq C\mathcal{C}_F(\XX,\Gamma,\nu T),
\end{equation}
where $\|\cdot\|_{T,p}$ is the norm defined in \eqref{def: temporal Morrey norm} and  $C>0$ is a constant independent of the parameters of the problem.
\end{lem}
\begin{proof}
We begin by recalling the form of the forcing term obtained in Remark \ref{rem: cancelacion torsion}:
\[
\begin{aligned}
\bm{F}[\vo_0,\vv_0]=&\frac{r}{h}\frac{\Gamma}{4\pi}\log(\sqrt{\nu t})\Omega^\eta\frac12\partial_s\left(|\TT_s|^2\right)\ea_{\varphi}-\vv^*\cdot\nabla\vo_0+\vo_0\cdot\nabla\vv_{L\text{-}O}+\vo_0\cdot\nabla\vv^*\\
&+\nu\left(\frac1r\Omega\eta_{R,r}+2\Omega_r\eta_{R,r}+\Omega\eta_{R,rr}-\frac{\ea_r\cdot\TT_s}{h}\Omega\eta_{R,r}\right)\TT+\nu\left(-\frac{\ea_r\cdot\TT_s}{h}\Omega_r^\eta\TT+\Omega^\eta\Delta\TT\right).
\end{aligned}
\]
We split the estimate into three classes of terms.

First, we consider the leading-order terms, namely those with singularity $(\nu t)^{-3/2}$:
\[
\bm{F}_1=-\vv^*\cdot\nabla\vo_0+\vo_0\cdot\nabla\vv_{L\text{-}O}-\nu\frac{\ea_r\cdot\TT_s}{h}\Omega_r\eta_R\TT.
\]
Using that $\rho e^{-\rho^2}\leq C e^{-c\rho^2}$ for some $0<c<1$ and $C>0$,  together with
\[
|\vo_0\cdot\nabla\vv_{L\text{-}O}|\leq C\frac{\Gamma}{\sqrt{\nu t}}\Omega^\eta|\TT_s|,\qquad |\vv^*\cdot\nabla\vo_0|\leq C|\vv^*|(\Omega^\eta|\TT_s|+\Omega^\eta_r),
\]
we obtain
\begin{equation}\label{cota F_1}
|\bm{F}_1|\leq C\Gamma\left(|\TT_s|(\nu+\Gamma)+|\vv^*|(1+\sqrt{\nu t}|\TT_s|)\right)\frac{e^{-c\frac{r^2}{\nu t}}}{(\nu t)^{3/2}}\eta_R.
\end{equation}
Then, recalling Proposition \ref{Lemma_v_o},
\[
|\vv^*|\leq C\Gamma\|\TT'\|_\infty\left(\log\left(1+\frac{r}{\sqrt{\nu t}}\right)+\frac{|A(l,r)|}{\|\TT'\|_\infty}+\left|\log(\|\TT'\|_\infty)\right|+1+\left|\log\bigl(\sqrt{\nu t}\bigr)\right|e^{-\frac{R^2}{4\nu t}}\right),
\]
and using that $\log(1+\rho)e^{-c\rho^2}\leq C e^{-c'\rho^2}$, for some $0<c'<c$ and $C>0$, we define $C^*$ so that
\begin{equation}\label{weighted v star}
|\vv^*|e^{-c\frac{r^2}{\nu t}}\leq \Gamma C^*e^{-c'\frac{r^2}{\nu t}},
\end{equation}
this is,
\begin{equation}\label{C estrella}
C^*:=C\left(\|A(l,r)\|_\infty+\|\TT'\|_\infty\left[1+\left|\log(\|\TT'\|_\infty)\right|\right]+\|\TT'\|_\infty\sup_{0<t<T}\left|\log\bigl(\sqrt{\nu t}\bigr)\right|e^{-\frac{R^2}{4\nu t}}\right).
\end{equation}
Hence, \eqref{cota F_1} and \eqref{weighted v star} yield
\[
|\bm{F}_1|\leq C\Gamma\left(\|\TT_s\|_\infty(\nu+\Gamma)+\Gamma C^*(1+\sqrt{\nu T}\|\TT_s\|_\infty)\right)\frac{e^{-c\frac{r^2}{\nu t}}}{(\nu t)^{3/2}}\eta_R.
\]
On the other hand, Proposition \ref{proposition Giga Miya convolucion calor} and Lemma \ref{lema 5.5} yield
\[
\left\|\int_0^t e^{\nu(t-\tau)\Delta}\frac{e^{-c\frac{r^2}{\nu \tau}}}{(\nu \tau)^{3/2}}\eta_R\,{\rm d}\tau\right\|_p\leq C\int_0^t[\nu(t-\tau)]^{-\frac{1}{2}(\frac{3}{q}-\frac{3}{p})}(\nu \tau)^{-(\frac{3}{2}-\frac{3}{2q})}\,{\rm d}\tau.
\]
Then, choosing $q$ so that the time singularities are integrable, more precisely,
\[
q=
\begin{cases}
\frac32, & \text{if } \frac32\leq p<\infty,\\
2, & \text{if } p=\infty,
\end{cases}
\]
we get
\[
\left\|\int_0^t e^{\nu(t-\tau)\Delta}\frac{e^{-c\frac{r^2}{\nu \tau}}}{(\nu \tau)^{3/2}}\eta_R\,{\rm d}\tau\right\|_p\leq \frac{C}{\nu}(\nu t)^{\frac{3}{2p}-\frac12}.
\]
Combining the previous estimates, we obtain
\[
\left\|\int_0^t e^{\nu(t-\tau)\Delta}\bm{F}_1(\tau)\,{\rm d}\tau\right\|_{T,p}\leq C\frac{\Gamma}{\nu}\left(\|\TT_s\|_\infty(\nu+\Gamma)+\Gamma C^*(1+\sqrt{\nu T}\|\TT_s\|_\infty)\right).
\]

Second, we consider the terms which are written in divergence form,
\[
\vo_0\cdot\nabla\vv^*+\nu\Omega^\eta\Delta\TT.
\]
Since $\nabla\cdot\vo_0=0$, we have
\[
(\vo_0\cdot\nabla)\vv^*=\nabla\cdot(\vo_0\otimes\vv^*).
\]
Moreover,
\[
\nu\Omega^\eta\Delta\TT=\nu\nabla\cdot(\Omega^\eta\nabla\TT)-\nu\nabla\Omega^\eta\cdot\nabla\TT.
\]
The last term is of the same type as the terms in $\bm{F}_1$, and is estimated analogously. Hence we focus on
\[
\bm{F}_2=\nabla\cdot(\vo_0\otimes\vv^*)+\nu\nabla\cdot(\Omega^\eta\nabla\TT).
\]
Using Proposition \ref{proposition Giga Miya convolucion calor} together with \eqref{weighted v star}, we get
\[
\begin{aligned}
\left\|\int_0^t e^{\nu(t-\tau)\Delta}\bm{F}_2(\tau)\,{\rm d}\tau\right\|_p
&\leq\int_0^t\left\|\nabla e^{\nu(t-\tau)\Delta}\big(\vo_0\otimes\vv^*+\nu\Omega^\eta\nabla\TT\big)(\tau)\right\|_p\,{\rm d}\tau\\
&\leq C\Gamma(\Gamma C^*+\nu\|\TT_s\|_\infty)\int_0^t[\nu(t-\tau)]^{-\frac12-\frac12(\frac3q-\frac3p)}\left\|\frac{e^{-c\frac{r^2}{\nu \tau}}}{\nu \tau}\eta_R\right\|_q\,{\rm d}\tau.
\end{aligned}
\]
Then, by Lemma \ref{lema 5.5},
\[
\left\|\int_0^t e^{\nu(t-\tau)\Delta}\bm{F}_2(\tau)\,{\rm d}\tau\right\|_p\leq C\Gamma(\Gamma C^*+\nu\|\TT_s\|_\infty)\int_0^t[\nu(t-\tau)]^{-\frac12-\frac12(\frac3q-\frac3p)}(\nu\tau)^{-(1-\frac{3}{2q})}\,{\rm d}\tau.
\]
We choose $q$ so that the time singularities are integrable. More precisely,
\[
q=
\begin{cases}
p, & \text{if } \frac32\leq p<\infty,\\
4, & \text{if } p=\infty.
\end{cases}
\]
With this choice,
\[
\left\|\int_0^t e^{\nu(t-\tau)\Delta}\bm{F}_2(\tau)\,{\rm d}\tau\right\|_{T,p}\leq C\frac{\Gamma}{\nu}(\Gamma C^*+\nu\|\TT_s\|_\infty).
\]

Finally, the non-divergence terms isolated in Remark \ref{rem: cancelacion torsion} have been rewritten as
\[
\frac{r}{h}\frac{\Gamma}{4\pi}\log(\sqrt{\nu t})\Omega^\eta\frac12\partial_s(|\TT_s|^2)\ea_\varphi,
\]
which is less singular than $\bm{F}_1$ and yields a contribution of order $O(|\log(\nu T)|\sqrt{\nu T})$ in the bound for $\|\tilde{\vo}_F\|_{T,p}$, under the geometric bounds assumed above. The remaining cutoff terms are also less singular in $\nu t$. Combining the estimates above and using $\Gamma/\nu=O(1)$, we obtain
\[
\|\tilde{\vo}_F\|_{T,p}\leq C\Gamma\left(\|\TT_s\|_\infty+\frac{\Gamma}{\nu}C^*+O(|\log(\nu T)|\sqrt{\nu T})\right).
\]
The explicit form of $\mathcal{C}_F$ in the statement follows from \eqref{C estrella}.
\end{proof}
At this point, we are ready to prove the existence result. The estimate of the forcing term obtained in Lemma \ref{estimacion 3/2 termino de fuerza} provides the inhomogeneous contribution in the Duhamel formula. The remaining terms will be controlled by a fixed point argument in the time-weighted Morrey spaces introduced in \eqref{def: temporal Morrey norm}, using the heat semigroup and Biot--Savart estimates recalled above.
\begin{thm}\label{Teorema vorticidad}
Let $\XX\in C^3$ be an open curve in $\mathbb R^3$, evolving under the binormal flow, with initial data $\XX(s,0)=\XX_0(s)$. Assume that the radius $R$ of the tubular neighborhood given by \eqref{tubular intro} is well defined and that there exists $c_0>0$ such that the chord-arc condition \eqref{uniform_chord_arc_condition_intro} holds uniformly in time on $(0,T)$. Assume also that $\Gamma/\nu$ and $\nu T$ are sufficiently small.  Then there exists a solution $\bm{\omega}(\xx,t)$ to \eqref{vorticity_NS}, defined on $(0,T)$, such that, for $\xx\in\mathcal{T}_R(t)$, its expression in the local coordinates $(r,\varphi,s)$ is given by
\[
\vo(r,\varphi, s,t)=\frac{\Gamma}{4\pi}\frac{e^{-\frac{r^2}{4\nu t}}}{\nu t}\eta_R(r)\TT(s,t)+\tilde{\vo}(r,\varphi, s,t),
\]
for $\xx\in\mathcal{T}_R(t)$. Moreover, for every $\frac32\leq p\leq\infty$, the perturbation satisfies
\begin{equation}\label{cota p vorticidad}
\sup_{0<t<T}(\nu t)^{\frac12-\frac{3}{2p}}\|\tilde{\vo}(t)\|_p\leq C\mathcal{C}_F(\XX,\Gamma,\nu T),
\end{equation}
where $\mathcal{C}_F(\XX,\Gamma,\nu T)$ is defined in \eqref{constante explicita C_F}. Furthermore, for every $\frac32\leq p<3$,
\begin{equation}\label{convergencia omega tilde p menor 3}
\|\tilde{\vo}(t)\|_p\leq C\mathcal{C}_F(\XX,\Gamma,\nu T)(\nu t)^{-\frac12+\frac{3}{2p}}\longrightarrow0\qquad\text{as }t\to0.
\end{equation}
In addition, the velocity induced by the perturbation satisfies
\begin{equation}\label{cota infinito velocidad}
\sup_{0<t<T}\|\tilde{\vv}(t)\|_\infty\leq C\mathcal{C}_F(\XX,\Gamma,\nu T).
\end{equation}
\end{thm}
\begin{proof}
We consider the map
\[
\Phi(\tilde{\vo})(t):=\tilde{\vo}_F(t)+\int_0^t e^{\nu(t-\tau)\Delta}\left(NL[\tilde{\vo},\tilde{\vv}]+L[\vo_0,\tilde{\vo}]\right)(\tau)\,{\rm d}\tau,
\]
where $\tilde{\vo}_F$ is defined in \eqref{definicion omega_F} and $\tilde{\vv}=K*\tilde{\vo}$. We shall write the nonlinear and linear terms in divergence form. For two divergence-free vector fields $\vo$ and $\vv$, define
\[
\bm{W}^i(\vo,\vv):=\omega^i\vv-v^i\vo.
\]
Then
\[
\partial_i\bm{W}^i(\vo,\vv)=(\vo\cdot\nabla)\vv-(\vv\cdot\nabla)\vo.
\]
Hence the map $\Phi$ can be written as
\begin{equation}\label{mapa Phi}
\Phi(\tilde{\vo})(t):=\tilde{\vo}_F(t)+\int_0^t\partial_i e^{\nu(t-\tau)\Delta}\bm{J}^i(\tilde{\vo},\vo_0)(\tau)\,{\rm d}\tau,
\end{equation}
where
\[
\bm{J}^i(\tilde{\vo},\vo_0):=\bm{W}^i(\tilde{\vo},K*\tilde{\vo})+\bm{W}^i(\tilde{\vo},K*\vo_0)+\bm{W}^i(\vo_0,K*\tilde{\vo}).
\]
Let us fix $q$ such that
\[
\frac32<q<3.
\]
We work in the Banach space
\[
X:=X_{T,q}\cap X_{T,2q},\qquad \|f\|_X:=\max\left\{\|f\|_{T,q},\|f\|_{T,2q}\right\},
\]
where the norms $\|\cdot\|_{T,p}$ are defined in \eqref{def: temporal Morrey norm}. The first part of the proof is devoted to proving that $\Phi$ is a contraction in a ball of $X$, provided that $\Gamma/\nu$ and $\nu T$ are sufficiently small. Once the fixed point has been obtained in $X$, we return to the integral formulation to recover the remaining estimates in the endpoint case $p=3/2$ and in the range $6\leq p\leq\infty$.

\emph{Fixed point in the space $X$.}
Let $p\in\{q,2q\}$. By Proposition \ref{proposition Giga Miya convolucion calor}, we have\small
\begin{equation}\label{Picard vorticity 2}
\begin{aligned}
\left\|\nabla e^{\nu(t-\tau)\Delta}\bm{J}(\tilde{\vo},\vo_0)(\tau)\right\|_{p}\hspace{-1.8mm}\leq C[\nu(t-\tau)]^{-\frac{1}{2}-\frac{3}{2}(\frac{1}{q}-\frac{1}{p})}
\left(
\|\tilde{\vo}\|_q\|K*\tilde{\vo}\|_{\infty}
+\|\tilde{\vo}\|_q\|K*\vo_0\|_{\infty}
+\|\vo_0\|_q\|K*\tilde{\vo}\|_{\infty}
\right).
\end{aligned}
\end{equation}\normalsize
Using Proposition \ref{proposition Giga Miya convolucion biot savart}, we estimate
\begin{align*}
\|\tilde{\vo}\|_q\|K*\tilde{\vo}\|_\infty&\leq C\|\tilde{\vo}\|_q^{\frac{2q}{3}}\|\tilde{\vo}\|_{2q}^{2-\frac{2q}{3}},\\
\|\tilde{\vo}\|_q\|K*\vo_0\|_\infty&\leq C\|\tilde{\vo}\|_q\|\vo_0\|_q^{\frac{2q}{3}-1}\|\vo_0\|_{2q}^{2-\frac{2q}{3}},\\
\|\vo_0\|_q\|K*\tilde{\vo}\|_\infty&\leq C\|\vo_0\|_q\|\tilde{\vo}\|_q^{\frac{2q}{3}-1}\|\tilde{\vo}\|_{2q}^{2-\frac{2q}{3}}.
\end{align*}
Moreover, by Lemma \ref{lema 5.5},
\[
\|\vo_0(\tau)\|_q\leq C\Gamma(\nu\tau)^{-1+\frac{3}{2q}},\qquad \|\vo_0(\tau)\|_{2q}\leq C\Gamma(\nu\tau)^{-1+\frac{3}{4q}}.
\]
Therefore,
\begin{align*}
\|\tilde{\vo}\|_q\|K*\tilde{\vo}\|_\infty&\leq C\|\tilde{\vo}\|_{T,q}^{\frac{2q}{3}}\|\tilde{\vo}\|_{T,2q}^{2-\frac{2q}{3}}(\nu\tau)^{-\frac12+\frac{3}{2q}},\\
\|\tilde{\vo}\|_q\|K*\vo_0\|_\infty&\leq C\Gamma\|\tilde{\vo}\|_{T,q}(\nu\tau)^{-1+\frac{3}{2q}},\\
\|\vo_0\|_q\|K*\tilde{\vo}\|_\infty&\leq C\Gamma\|\tilde{\vo}\|_{T,q}^{\frac{2q}{3}-1}\|\tilde{\vo}\|_{T,2q}^{2-\frac{2q}{3}}(\nu\tau)^{-1+\frac{3}{2q}}.
\end{align*}
Combining these estimates with \eqref{mapa Phi}, we get
\begin{align*}
\|\Phi(\tilde{\vo})(t)\|_p\leq& \|\tilde{\vo}_F(t)\|_p+C\|\tilde{\vo}\|_{T,q}^{\frac{2q}{3}}\|\tilde{\vo}\|_{T,2q}^{2-\frac{2q}{3}}\int_0^t[\nu(t-\tau)]^{-\frac12-\frac32(\frac1q-\frac1p)}(\nu\tau)^{-\frac12+\frac{3}{2q}}\,{\rm d}\tau\\
&+C\Gamma\left(\|\tilde{\vo}\|_{T,q}+\|\tilde{\vo}\|_{T,q}^{\frac{2q}{3}-1}\|\tilde{\vo}\|_{T,2q}^{2-\frac{2q}{3}}\right)\int_0^t[\nu(t-\tau)]^{-\frac12-\frac32(\frac1q-\frac1p)}(\nu\tau)^{-1+\frac{3}{2q}}\,{\rm d}\tau.
\end{align*}
Since $3/2<q<3$ and $p\in\{q,2q\}$, both time integrals are finite. Multiplying by $(\nu t)^{\frac12-\frac{3}{2p}}$, taking the supremum over $0<t<T$, and using Lemma \ref{estimacion 3/2 termino de fuerza}, we obtain
\begin{align*}
\|\Phi(\tilde{\vo})\|_{T,p}\leq& C\mathcal{C}_F+\frac{C}{\nu}(\nu T)^{\frac12}\|\tilde{\vo}\|_{T,q}^{\frac{2q}{3}}\|\tilde{\vo}\|_{T,2q}^{2-\frac{2q}{3}}\\
&+C\frac{\Gamma}{\nu}\left(\|\tilde{\vo}\|_{T,q}+\|\tilde{\vo}\|_{T,q}^{\frac{2q}{3}-1}\|\tilde{\vo}\|_{T,2q}^{2-\frac{2q}{3}}\right).
\end{align*}
Taking the maximum over $p\in\{q,2q\}$, this gives
\begin{equation}\label{estimacion Phi X}
\|\Phi(\tilde{\vo})\|_X\leq C\mathcal{C}_F+\frac{C}{\nu}(\nu T)^{\frac12}\|\tilde{\vo}\|_X^2+C\frac{\Gamma}{\nu}\|\tilde{\vo}\|_X.
\end{equation}
With these scheme let us define
\[
X=X_{T,q}\cap X_{T,2q},\qquad \|f\|_X=\max\left\{\|f\|_{T,q},\|f\|_{T,2q}\right\}, \qquad B_M:=\{f\in X:\|f\|_X\leq M\},
\]
and choose
\[
M:=2C\mathcal{C}_F.
\]
If $\tilde{\vo}\in B_M$, then \eqref{estimacion Phi X} gives
\[
\|\Phi(\tilde{\vo})\|_X
\leq \frac{M}{2}+\frac{C}{\nu}(\nu T)^{\frac12}M^2+C\frac{\Gamma}{\nu}M.
\]
Hence $\Phi(B_M)\subset B_M$ provided that
\begin{equation}\label{smallness condition}
\frac{C}{\nu}(\nu T)^{\frac12}M+C\frac{\Gamma}{\nu}\leq \frac12.
\end{equation}
The same estimates applied to the difference of two elements of $B_M$ yield
\[
\|\Phi(\tilde{\vo}_1)-\Phi(\tilde{\vo}_2)\|_X
\leq
\left(
\frac{C}{\nu}(\nu T)^{\frac12}M+C\frac{\Gamma}{\nu}
\right)
\|\tilde{\vo}_1-\tilde{\vo}_2\|_X.
\]
Thus, under the same smallness condition \eqref{smallness condition}, $\Phi$ is a contraction on $B_M$. Following the fixed point argument of \cite{GigaMiyakawa} and the references therein, condition \eqref{smallness condition} yields a unique fixed point
\[
(\nu t)^{\frac12-\frac{3}{2q}}\tilde{\vo}\in BC((0,T);\mathcal M^q),
\qquad
(\nu t)^{\frac12-\frac{3}{4q}}\tilde{\vo}\in BC((0,T);\mathcal M^{2q}),
\]
with
\[
\|\tilde{\vo}\|_{T,q}+\|\tilde{\vo}\|_{T,2q}\leq C\mathcal{C}_F.
\]
Then, Proposition \ref{proposition Giga Miya convolucion biot savart} yields
\begin{equation}\label{cota infinito velocidad prueba}
\|K*\tilde{\vo}(t)\|_\infty\leq C\mathcal{C}_F,\qquad 0<t<T.
\end{equation}
Since the construction above can be performed for any $q\in(\frac32,3)$, this gives the estimate for every $\frac32<p<6$.

\emph{Recovery of the remaining estimates.}
We now use the fixed point obtained in the previous step to recover the endpoint $p=\frac32$ and the range $6\leq p\leq\infty$. We start from the integral formulation
\[
\tilde{\vo}(t)=\tilde{\vo}_F(t)+\int_0^t\partial_i e^{\nu(t-\tau)\Delta}\bm{J}^i(\tilde{\vo},\vo_0)(\tau)\,{\rm d}\tau.
\]
First, we consider the endpoint $p=\frac32$. Applying Proposition \ref{proposition Giga Miya convolucion calor} with $q=p=\frac32$, together with Lemma \ref{lema 5.5} and \eqref{cota infinito velocidad prueba}, we get
\begin{align*}
\|\tilde{\vo}(t)\|_{\frac32}\leq\|&\tilde{\vo}_F(t)\|_{\frac32}+C\int_0^t[\nu(t-\tau)]^{-\frac12}\|\tilde{\vo}(\tau)\|_{\frac32}\|K*\tilde{\vo}(\tau)\|_\infty\,{\rm d}\tau\\
&+C\int_0^t[\nu(t-\tau)]^{-\frac12}\|\tilde{\vo}(\tau)\|_{\frac32}\|K*\vo_0(\tau)\|_\infty\,{\rm d}\tau\\
&+C\int_0^t[\nu(t-\tau)]^{-\frac12}\|\vo_0(\tau)\|_{\frac32}\|K*\tilde{\vo}(\tau)\|_\infty\,{\rm d}\tau.
\end{align*}
Using
\[
\|K*\tilde{\vo}(\tau)\|_\infty\leq C\mathcal{C}_F,\qquad \|K*\vo_0(\tau)\|_\infty\leq C\Gamma(\nu\tau)^{-\frac12},\qquad \|\vo_0(\tau)\|_{\frac32}\leq C\Gamma,
\]
and the definition of $\|\tilde{\vo}\|_{T,\frac32}$, we obtain
\begin{align*}
\|\tilde{\vo}(t)\|_{\frac32}\leq &C\mathcal{C}_F(\nu t)^{\frac12}+C\mathcal{C}_F\|\tilde{\vo}\|_{T,\frac32}\int_0^t[\nu(t-\tau)]^{-\frac12}(\nu\tau)^{\frac12}\,{\rm d}\tau\\
&+C\Gamma\|\tilde{\vo}\|_{T,\frac32}\int_0^t[\nu(t-\tau)]^{-\frac12}\,{\rm d}\tau+C\Gamma\mathcal{C}_F\int_0^t[\nu(t-\tau)]^{-\frac12}\,{\rm d}\tau.
\end{align*}
Multiplying by $(\nu t)^{-\frac12}$ and taking the supremum in $0<t<T$, we get
\[
\|\tilde{\vo}\|_{T,\frac32}\leq C\mathcal{C}_F+\frac{C}{\nu}(\nu T)^{\frac12}\mathcal{C}_F\|\tilde{\vo}\|_{T,\frac32}+C\frac{\Gamma}{\nu}\|\tilde{\vo}\|_{T,\frac32}+C\frac{\Gamma}{\nu}\mathcal{C}_F.
\]
Hence, for sufficiently small $\Gamma/\nu$ and $\nu T$,
\begin{equation}\label{cota tres medios vorticidad}
\|\tilde{\vo}\|_{T,\frac32}\leq C\mathcal{C}_F.
\end{equation}

We now consider the range $6\leq p\leq\infty$. Let $r\in(3,6)$ be such that $r\leq p$. Applying Proposition \ref{proposition Giga Miya convolucion calor}, Lemma \ref{lema 5.5}, and \eqref{cota infinito velocidad prueba}, we obtain
\begin{align*}
\|\tilde{\vo}(t)\|_p\leq&\|\tilde{\vo}_F(t)\|_p+C\mathcal{C}_F\|\tilde{\vo}\|_{T,r}\int_0^t[\nu(t-\tau)]^{-\frac12-\frac32(\frac1r-\frac1p)}(\nu\tau)^{-\frac12+\frac{3}{2r}}\,{\rm d}\tau\\
&+C\Gamma\|\tilde{\vo}\|_{T,r}\int_0^t[\nu(t-\tau)]^{-\frac12-\frac32(\frac1r-\frac1p)}(\nu\tau)^{-1+\frac{3}{2r}}\,{\rm d}\tau\\
&+C\Gamma\mathcal{C}_F\int_0^t[\nu(t-\tau)]^{-\frac12-\frac32(\frac1r-\frac1p)}(\nu\tau)^{-1+\frac{3}{2r}}\,{\rm d}\tau.
\end{align*}
The time integrals are finite because
\[
\frac12+\frac32\left(\frac1r-\frac1p\right)<1,
\]
which follows from $r>3$ and $r\leq p$. Multiplying by $(\nu t)^{\frac12-\frac{3}{2p}}$ and taking the supremum in $0<t<T$, we get
\[
\|\tilde{\vo}\|_{T,p}\leq C\mathcal{C}_F+\frac{C}{\nu}(\nu T)^{\frac12}\mathcal{C}_F\|\tilde{\vo}\|_{T,r}+C\frac{\Gamma}{\nu}\left(\|\tilde{\vo}\|_{T,r}+\mathcal{C}_F\right).
\]
Since $3<r<6$, the estimate for $\|\tilde{\vo}\|_{T,r}$ was already obtained in the previous step. Therefore, for $6\leq p\leq\infty$,
\begin{equation}\label{cota p grande vorticidad}
\|\tilde{\vo}\|_{T,p}\leq C\mathcal{C}_F.
\end{equation}
Combining the estimates for $\frac32<p<6$, \eqref{cota tres medios vorticidad}, and \eqref{cota p grande vorticidad}, we obtain
\[
\|\tilde{\vo}\|_{T,p}\leq C\mathcal{C}_F,\qquad \frac32\leq p\leq\infty.
\]
This proves \eqref{cota p vorticidad}.

Finally, \eqref{cota infinito velocidad prueba} gives
\[
\sup_{0<t<T}\|K*\tilde{\vo}(t)\|_\infty\leq C\mathcal{C}_F,
\]
which is \eqref{cota infinito velocidad}.

\emph{Vanishing as $t\to0$.}
If $\frac32\leq p<3$, then
\[
-\frac12+\frac{3}{2p}>0.
\]
Therefore, from \eqref{cota p vorticidad},
\[
\|\tilde{\vo}(t)\|_p\leq C\mathcal{C}_F(\nu t)^{-\frac12+\frac{3}{2p}}\longrightarrow0\qquad\text{as }t\to0.
\]
This proves \eqref{convergencia omega tilde p menor 3} and completes the proof.
\end{proof}
\section{Hasimoto solitons}\label{sec: moving filament}

In this section we focus on solutions of \eqref{vorticity_NS} when the initial data is concentrated on Hasimoto solitons. Hasimoto solitons is a family of open curves whose shape is preserved by the binormal flow while traslating and rotating with constant velocity. This allows us to follow the solution in local coordinates attached to the soliton. We use this explicit structure to find a suitable configuration of the geometric parameters for which the kinetic energy of the solution inside the physical core of the filament, a region that will be defined below, is effectively displaced. This displacement of energy is due to the nonlinear effects associated with the binormal flow.

The main point relies on choosing the parameters of the soliton so that the curvature remains $O(1)$, while the torsion is large enough to produce a macroscopic displacement during the time interval under consideration. To do this, we verify that this large torsion regime is still compatible with the tubular framework and with the estimates used in the fixed point argument. Finally, we show that the kinetic energy in the core contains a localized component which is transported along the filament with the Hasimoto soliton.
\subsection{Definition of the soliton and choice of parameters}
The \emph{Hasimoto soliton}, introduced in \cite{Hasimoto72}, is an explicit shape-preserving curve under the binormal flow. It arises when looking for solutions to \eqref{cubica schrodinger} that propagate steadily with constant velocity $2\tau_0$ along the filament and are asymptotically straight, namely
\[
\kappa(s,t)=|\TT_s(s,t)|\to0\qquad\text{as}\qquad s\to\infty.
\]
A straightforward computation (see \cite{Hasimoto72}) yields a solution of the form
\[
\Psi(s,t)=\kappa_{\lambda}(s,t)e^{i\tau_0 s},\qquad \kappa_{\lambda}(s,t)=2\lambda\operatorname{sech}\big(\lambda(s-2\tau_0 t)\big),
\]
so recalling the Hasimoto transform
\[
\psi(s,t)=\alpha(s,t)+i\beta(s,t)=\kappa(s,t)\exp\!\left(i\int_0^s\tau(s',t)\,ds'\right),
\]
we define $\XX_{\mathrm{sol}}(s,t)$ as the curve whose associated parallel frame satisfies \eqref{frenet_serret_frame} with
\[
\alpha(s,t)=\kappa_{\lambda}(s,t)\cos(\tau_0 s),\qquad \beta(s,t)=\kappa_{\lambda}(s,t)\sin(\tau_0 s).
\]
Notice that the soliton exhibits a localized bump structure, in which the curvature is concentrated on a length scale of order $\lambda^{-1}$ and propagates along the filament with speed $2\tau_0$.

Our goal is to construct a geometric configuration for which the solution obtained in Theorem \ref{Teorema vorticidad} exhibits a visible displacement during the time interval $(0,T)$. Since the soliton evolves under the physical binormal flow
\[
\XX_t=-\frac{\Gamma}{4\pi}\log(\sqrt{\nu t})\,\TT\wedge\TT_s,
\]
the natural time variable in Hasimoto's formulation differs from the
Navier--Stokes time scale by the logarithmic factor above. Undoing this
time rescaling, the center of the curvature bump propagates with velocity
\[
v_{{\rm prop}}(t)=\frac{\Gamma}{2\pi}\tau_0\log(\sqrt{\nu t}).
\]
Therefore, assuming that $\nu T\ll1$, its total displacement during the
time interval $(0,T)$ is
\[
\Delta s=\frac{\Gamma\tau_0}{2\pi}\int_{0}^{T}\left|\log(\sqrt{\nu t})\right|\,dt=\frac{\Gamma\tau_0T}{4\pi}\bigl(|\log(\nu T)|+1\bigr).
\]

Later, we shall choose a parameter $L>0$ so that a longitudinal interval of length
$2L$, centered at the curvature bump, contains almost all the localized
binormal energy. We therefore choose the torsion so that the center of the
bump travels precisely this distance during the interval $(0,T)$. Namely,
throughout this section we set
\begin{equation}\label{curvature and torsion order}
\lambda=1,\qquad
\tau_0=
\frac{8\pi L}
{\frac{\Gamma}{\nu} (\nu T)\bigl(|\log(\nu T)|+1\bigr)}\gg1.
\end{equation}
With this choice, the curvature bump is displaced longitudinally by a
distance $2L$ during the time interval $(0,T)$. 
\subsection{Tubular neighborhood of the Hasimoto soliton}

We now verify that the soliton with the parameters described above is geometrically admissible within the framework developed in Section \ref{sec: the local frame}. In order to apply Theorem \ref{Teorema vorticidad}, we first need to ensure that the Hasimoto soliton admits a regular tubular neighborhood with radius independent of the torsion.

The local condition follows directly from the curvature bound. Since
\[
\kappa_{\max}=2\lambda,
\]
the local radius remains of order $\lambda^{-1}$, and therefore is $O(1)$ under the assumption $\lambda=O(1)$.

It remains to verify the global non-self-intersection condition. By fixing the orientation so that the filament approaches the $z$-axis at infinity, in \cite{Hasimoto72} they prove that
\begin{equation}\label{expresion fisica de soliton hasimoto}
\begin{aligned}
\XX_{{\rm sol}}(s,t)\;&:\;z=s-\frac{2\mu}{\lambda}\tanh(\lambda(s-2\tau_0 t)),\qquad x+iy=re^{i\Theta},\\
\TT_{{\rm sol}}(s,t)\;&:\;\TT_z=1-2\mu\,{\rm sech}^2(\lambda(s-2\tau_0 t)),\qquad \TT_x+i\TT_y=-\lambda r\left(\tanh(\lambda(s-2\tau_0 t))-i\frac{\tau_0}{\lambda}\right)e^{i\Theta},
\end{aligned}
\end{equation}
where
\[
\mu=\frac{\lambda^2}{\lambda^2+\tau_0^2},\qquad r=\frac{2\mu}{\lambda}{\rm sech}(\lambda(s-2\tau_0 t)),\qquad \Theta=\tau_0 s+(\lambda^2-\tau_0^2)t.
\]
Notice that the relation between $z$ and $s$ can be rewritten as
\[
z-2\tau _{0}t=s-2\tau _{0}t-\frac{2\lambda }{\lambda ^{2}+\tau _{0}^{2}}
\tanh \left( \lambda (s-2\tau _{0}t\right) )\equiv G(s-2\tau _{0}t)
\]
and we can invert $G$, i.e. write
\[
s-2\tau _{0}t=G^{-1}(z-2\tau _{0}t)
\]
provided $\tau _{0}\geq \lambda $. This implies that the soliton translates
in the $z$ direction with a velocity $2\tau _{0}$ or, in other words, can be
written as
\[
x(z,t)+iy(z,t)=\tilde{r}(z-2\tau _{0}t)e^{i\tilde{\Theta }(z-2\tau
_{0}t)}e^{i(\lambda ^{2}+\tau _{0}^{2})t}
\]
with $(\tilde{r},\tilde{\Theta })=(r\circ G^{-1},\Theta \circ G^{-1})
$. This formula also implies that the soliton rotates with an angular
velocity $\omega =(\lambda ^{2}+\tau _{0}^{2})$. In addition, the Hasimoto soliton satisfies the following chord arc condition
\[
|\XX(s,t)-\XX(s_0,t)|\geq |z(s,t)-z(s_0,t)|\geq \frac{\tau_0^2-\lambda^2}{\tau_0^2+\lambda^2}|s-s_0|,
\]
and for sufficiently large torsion admits a regular tubular neighborhood $\mathcal{T}_R$ with radius
\[
R=O(1).
\]
\subsection{Uniformity of the estimates in the torsion parameter}\label{subsec:Uniformity of the estimates in the torsion parameter}

We next prove that the estimates involved in the fixed point argument remain uniform in the torsion parameter. This is the key point that allows us to consider the large torsion regime introduced above while preserving the framework of Section \ref{sec: existence result}.

\begin{lem}\label{cancelacion en el soliton para A}
Let $\XX$ be the Hasimoto soliton. Then the quantity
\[
A(l,|\hx|)=\int_{-l}^l\left(\int_0^s[\TT'(\xi)-\TT'(0)]\xi\,d\xi\right)\frac{ds}{(|\hx|^2+s^2)^{3/2}}
\]
can be bounded independently of the constant torsion parameter $\tau_0$ by
\[
|A(l,|\hx|)|\leq C\left(\frac{\|\partial_s|\TT_s|\|_\infty}{\|\TT_s\|_\infty}+\|\TT_s\|_\infty\right)=C'\lambda,
\]
where $C,C'>0$ are universal constants.
\end{lem}

\begin{proof}
Using the decomposition
\[
\TT'(s)=\alpha(s)\ea_1(s)+\beta(s)\ea_2(s),
\]
and \eqref{alpha_beta_trig}, we write
\[
\begin{aligned}
\TT'(s)-\TT'(0)=&\big(\kappa(s)-\kappa(0)\big)\big(\cos\theta(s)\ea_1(s)+\sin\theta(s)\ea_2(s)\big)\\
&+\kappa(0)\Big[\cos\theta(s)\big(\ea_1(s)-\ea_1(0)\big)+\sin\theta(s)\big(\ea_2(s)-\ea_2(0)\big)\Big]\\
&+\kappa(0)\Big[\big(\cos\theta(s)-\cos\theta(0)\big)\ea_1(0)+\big(\sin\theta(s)-\sin\theta(0)\big)\ea_2(0)\Big].
\end{aligned}
\]
This yields
\[
|A(l,|\hx|)|\leq C\left(\frac{\|\partial_s|\TT_s|\|_\infty}{\|\TT_s\|_\infty}+\|\TT_s\|_\infty\right)+\|\TT_s\|_\infty\left|\int_{-l}^l\left(\int_0^s(e^{i\theta(\xi)}-e^{i\theta(0)})\xi\,d\xi\right)\frac{ds}{(|\hat{\xx}|^2+s^2)^{3/2}}\right|.
\]
Recalling the explicit form $\theta(s)=\tau_0 s$ for the Hasimoto soliton and using Fubini's theorem, we obtain
\[
I:=\int_{-l}^l\left(\int_0^s(e^{i\tau_0\xi}-1)\xi\,d\xi\right)\frac{ds}{(|\hat{\xx}|^2+s^2)^{3/2}}=2i\int_0^l\sin(\tau_0\xi)\xi\left(\int_\xi^l\frac{ds}{(|\hat{\xx}|^2+s^2)^{3/2}}\right)d\xi.
\]
This quantity is uniformly bounded in $\tau_0$ and $|\hx|$. Indeed, as $|\hx|\to0$,
\[
I\longrightarrow i\int_0^l\frac{\sin(\tau_0\xi)}{\xi}\,d\xi-\frac{i}{l^2}\int_0^l\xi\sin(\tau_0\xi)\,d\xi,
\]
and both terms on the right-hand side are uniformly bounded in $\tau_0$. Therefore,
\[
|A(l,|\hx|)|\leq C\left(\frac{\|\partial_s|\TT_s|\|_\infty}{\|\TT_s\|_\infty}+\|\TT_s\|_\infty\right).
\]
and since for the Hasimoto soliton both terms are of order $\lambda$, the result follows.
\end{proof}

\begin{cor}\label{Corolario hasimoto}
For the Hasimoto soliton, the estimate of Lemma \ref{estimacion 3/2 termino de fuerza} is uniform with respect to the torsion parameter $\tau_0$.
\end{cor}

\begin{proof}
By Lemma \ref{lem:BS-in-s} and Remark \ref{rem: cancelacion torsion}, the only possible dependence of $|\vv^*|$ on $\tau_0$ enters through $|A(l,|\hx|)|$. Lemma \ref{cancelacion en el soliton para A} shows that, for the Hasimoto soliton, $|A(l,|\hx|)|$ can be bounded independently of $\tau_0$ and is bounded by $C\lambda$. Therefore the bound in Lemma \ref{estimacion 3/2 termino de fuerza} can be made uniform in $\tau_0$.
\end{proof}

\subsection{Macroscopic displacement of localized kinetic energy}

Now, we focus on the Hasimoto soliton in order to identify a portion of the filament core whose kinetic energy is transported by the binormal motion. Let us first recall the structure of the approximate solution constructed above:
\[
\bm{\omega}=\bm{\omega}_0+\tilde{\bm{\omega}},
\]
where, in tubular coordinates, the leading order term is given by
\[
\bm{\omega}_0(r,\varphi,s,t)=\frac{\Gamma}{4\pi}\frac{e^{-\frac{r^2}{4\nu t}}}{\nu t}\eta_R(r)\TT(s,t).
\]
The corresponding leading velocity field has the form
\[
\bm{v}_0=\bm{v}_{L\text{-}O}+\XX_t+\bm{v}^*,
\]
with
\[
\bm{v}_{L\text{-}O}(r,\varphi,s,t)=\frac{\Gamma}{2\pi}\frac{1}{r}\left(1-e^{-r^2/(4\nu t)}\right)\ea_\varphi(\varphi,s,t),
\]
and
\[
\XX_t(s,t)=-\frac{\Gamma}{4\pi}\log(\sqrt{\nu t})\,\TT(s,t)\wedge\TT_s(s,t).
\]

\subsubsection*{The physical core region}

We define the core through two physical conditions. The first one is a high-vorticity condition, which identifies the region where the Lamb--Oseen vorticity is the main contribution in the total vorticity. The second one is a velocity condition, which selects the region where the Lamb--Oseen velocity is smaller than binormal motion. More precisely, we set
\begin{equation}\label{core region label 0}
D_{M_t}=\{\bm{x}\in\RR^3:\ |\bm{\omega}(\bm{x},t)|\geq M_t\},\qquad D_{N_t}=\{\bm{x}\in\RR^3:\ |\bm{v}(\bm{x},t)|\leq N_t\},
\end{equation}
and define
\[
D_t=D_{M_t}\cap D_{N_t}.
\]
First, in order to set $D_t$, we recall that since $\lambda=1$, due to Theorem \ref{Teorema vorticidad}, the perturbation satisfies
\[
\|\tilde{\bm{\omega}}\|_\infty=O\bigl(\Gamma(\nu t)^{-1/2}\bigr).
\]
Then, since the leading Lamb--Oseen vorticity has size
\[
\|\bm{\omega}_0\|_\infty\sim \frac{\Gamma}{\nu t}e^{-r^2/(4\nu t)},
\]
the choice
\[
M_t\sim \frac{\Gamma}{\nu t}
\]
selects the region where the leading term $\bm{\omega}_0$ is still of order $(\nu t)^{-1}$. In tubular coordinates this amounts to requiring that the Gaussian factor has not decayed significantly, namely
\[
e^{-r^2/(4\nu t)}\sim 1.
\]
Thus the high-vorticity region is contained at transverse distances of order
\[
r\lesssim \sqrt{\nu t},
\]
which is the viscous scale associated with the Lamb--Oseen core.

Now, concerning the velocity condition, we choose $N_t$ comparable to the maximum of the binormal velocity 
\[
N_t\sim \frac{\Gamma}{4\pi}|\log(\sqrt{\nu t})|\|\TT_s\|_{\infty}= \frac{\Gamma}{2\pi}|\log(\sqrt{\nu t})|,
\]
using that $\lambda=1$. In tubular coordinates, this condition is translated into
\[
\frac{1}{r}\left(1-e^{-r^2/(4\nu t)}\right)\lesssim |\log(\sqrt{\nu t})|,
\]
which may in principle select more than one radial region. Nevertheless, the relevant component for the core is the inner one, since it is the one that intersects the high-vorticity region $D_{M_t}$. In this regime, $r/\sqrt{\nu t}\ll1$, and hence
\[
1-e^{-r^2/(4\nu t)}\sim \frac{r^2}{4\nu t}.
\]
Therefore, the velocity condition yields
\[
r\lesssim \nu t\,|\log(\nu t)|.
\]
Since
\[
\nu t\,|\log(\nu t)|\ll\sqrt{\nu t}
\]
for sufficiently small $\nu t$, this condition is more restrictive than the high-vorticity condition. Consequently, the intersection $D_t=D_{M_t}\cap D_{N_t}$ is, up to universal multiplicative constants, described by the tubular region
\begin{equation}\label{label-core-region}
D_t=\left\{\xx\in\mathbb{R}^3:
\inf_{s\in\mathbb{R}}|\xx-\XX(s,t)|
\leq r_c(t)\right\},
\qquad
r_c(t)\sim \nu t\,|\log(\nu t)|.
\end{equation}
\subsubsection*{Core energy and leading terms}
Let $I\subset\mathbb{R}$ be a bounded interval. We define the portion of the physical core lying over $I$ by
\[
D_t(I):=\left\{\xx=\XX(s,t)+r\ea_r(s,\varphi, t):
s\in I,\quad r\leq r_c(t)\right\}.
\]
In tubular coordinates around the Hasimoto soliton, the energy inside this portion of physical core is
\begin{equation}\label{enegy core total}
E_{\mathrm{core}}(I,t)
=\frac12\int_I\int_0^{2\pi}\int_0^{r_c(t)}
|\bm{v}(r,\varphi,s,t)|^2h(r,\varphi,s,t)\,r\,dr\,d\varphi\,ds.
\end{equation}
Since $\kappa=O(1)$ and $r_c(t)\ll1$ we have that $h\sim 1$, so using the decomposition
\[
\bm{v}_0=\bm{v}_{L\text{-}O}+\XX_t+\bm{v}^*,
\]
we obtain
\[
E_{\mathrm{core}}(I,t)
\sim\frac12\int_I\int_0^{2\pi}\int_0^{r_c(t)}
\left|
\frac{\Gamma}{8\pi\nu t}r\,\ea_\varphi(\varphi,s,t)
+\XX_t(s,t)+\bm{v}^*(r,\varphi,s,t)
\right|^2r\,dr\,d\varphi\,ds.
\]
Expanding the square, we obtain
\[
\begin{aligned}
&\left|\frac{\Gamma}{8\pi\nu t}r\,\ea_\varphi+\XX_t+\bm{v}^*\right|^2\\
&=\left(\frac{\Gamma}{8\pi\nu t}\right)^2r^2+\left(\frac{\Gamma}{4\pi}\right)^2|\log(\sqrt{\nu t})|^2|\TT\wedge\TT_s|^2-\frac{\Gamma^2}{16\pi^2\nu t}r\log(\sqrt{\nu t})\,\ea_\varphi\cdot(\TT\wedge\TT_s)\\
&\quad+|\bm{v}^*|^2+\frac{\Gamma}{4\pi\nu t}r\,\ea_\varphi\cdot\bm{v}^*-\frac{\Gamma}{2\pi}\log(\sqrt{\nu t})\,(\TT\wedge\TT_s)\cdot\bm{v}^*.
\end{aligned}
\]
Since, by Proposition \ref{Lemma_v_o} and Corollary \ref{Corolario hasimoto}, $\|\vv^*\|_\infty\sim\Gamma$, the first three terms above are the leading-order ones in the core. Indeed, since $r\leq r_c\sim \nu t|\log(\nu t)|$, we have
\[
\left(\frac{\Gamma}{8\pi\nu t}\right)^2r^2=O\bigl(\Gamma^2|\log(\nu t)|^2\bigr),
\]
while
\[
\left(\frac{\Gamma}{4\pi}\right)^2|\log(\sqrt{\nu t})|^2|\TT\wedge\TT_s|^2=O\bigl(\Gamma^2|\log(\nu t)|^2\bigr),
\]
and
\[
\frac{\Gamma^2}{16\pi^2\nu t}r|\log(\sqrt{\nu t})|\,|\ea_\varphi\cdot(\TT\wedge\TT_s)|=O\bigl(\Gamma^2|\log(\nu t)|^2\bigr).
\]
Nevertheless, the mixed Lamb--Oseen/binormal term cancels after integration in $\varphi$, since
\[
\ea_\varphi(\varphi,s,t)\cdot(\TT(s,t)\wedge\TT_s(s,t))=\kappa(s,t)\cos(\varphi-\tau_0s).
\]
Therefore, at leading order after integration over the tube,
\begin{equation}\label{energy core binormal}
E_{\mathrm{core}}(I,t)\sim E_{\mathrm{core}}^{L\text{-}O}(I,t)+E_{\mathrm{core}}^{BF}(I,t),
\end{equation}
where
\begin{equation}\label{energy core lamb oseen}
E_{\mathrm{core}}^{L\text{-}O}(I,t)=\frac12\int_I\int_0^{2\pi}\int_0^{r_c}\left(\frac{\Gamma}{8\pi\nu t}\right)^2r^3\,dr\,d\varphi\,ds,
\end{equation}
and
\[
E_{\mathrm{core}}^{BF}(I,t)=\frac12\int_I\int_0^{2\pi}\int_0^{r_c}\left(\frac{\Gamma}{4\pi}\right)^2|\log(\sqrt{\nu t})|^2|\TT\wedge\TT_s|^2\,r\,dr\,d\varphi\,ds.
\]
\subsubsection*{Localization of the binormal energy}

We now focus on the kinetic energy associated with the binormal transport. For any bounded interval $I\subset\mathbb{R}$, its leading-order contribution inside the core is
\[
E_{\mathrm{core}}^{BF}(I,t)\sim\frac12\int_I\int_0^{2\pi}\int_0^{r_c(t)}\left(\frac{\Gamma}{4\pi}\right)^2\left|\log(\sqrt{\nu t})\right|^2\kappa(s,t)^2\,r\,dr\,d\varphi\,ds.
\]
For the Hasimoto soliton with $\lambda=1$, the curvature is localized around the center $s=s_c(t)$ and is given by
\[
\kappa(s,t)=2\operatorname{sech}(s-s_c(t)).
\]
We consider the moving interval
\begin{equation}\label{label-moving-interval}
I_t=[s_c(t)-L,s_c(t)+L],
\end{equation}
where $L=O(1)$ is fixed. Direct integration gives
\[
\int_{I_t}\kappa(s,t)^2\,ds=8\tanh L=8\bigl(1+O(e^{-2L})\bigr),
\]
whereas
\[
\int_{\mathbb{R}\setminus I_t}\kappa(s,t)^2\,ds=8(1-\tanh L)=O(e^{-2L}).
\]
For instance, for $L=2$, more than $96\%$ of the binormal flow energy is concentrated in $I_t$. Therefore,
\[
\frac{E_{\mathrm{core}}^{BF}(\mathbb{R}\setminus I_t,t)}{E_{\mathrm{core}}^{BF}(I_t,t)}\ll1
\]
with
\[
E_{\mathrm{core}}^{BF}( I_t,t)=O\bigl(\Gamma^2(\nu t)^2|\log(\nu t)|^4\bigr).
\]

\subsubsection*{Comparison with the Lamb--Oseen background}

Over the same interval $I_t$, the Lamb--Oseen contribution is
\[
E_{\mathrm{core}}^{L\text{-}O}(I_t,t)\sim\frac12\int_{I_t}\int_0^{2\pi}\int_0^{r_c(t)}\left(\frac{\Gamma}{8\pi\nu t}\right)^2r^3\,dr\,d\varphi\,ds.
\]
Since $|I_t|=2L=O(1)$ and $r_c(t)\sim\nu t|\log(\nu t)|$, we obtain
\[
E_{\mathrm{core}}^{L\text{-}O}(I_t,t)\sim\Gamma^2(\nu t)^2|\log(\nu t)|^4\sim E_{\mathrm{core}}^{BF}(I_t,t)
\]
Notice that choosing $r_c =\delta \nu t|\log{\sqrt{\nu t}}|$ the term $E_{\mathrm{core}}^{L\text{-}O}(I_t,t)$ becomes small with respect to $E_{\mathrm{core}}^{BF}(I_t,t)$ with $\delta<1$ small.
Then, while Lamb--Oseen contribution forms a homogeneous background (i.e. with a density independent of $s$) along the physical core, the binormal contribution is localized near $s=s_c(t)$ and is transported with the Hasimoto soliton. Consequently, the energy inside $D_t(I_t)$ contains a localized component, comparable in size to the Lamb--Oseen background, which propagates along the filament.

Our discussion is consistent with the recent numerical results of Sterkers and Krstulovic \cite{SterkersKrstulovic2026}. Using three-dimensional Navier--Stokes simulations in a high-Reynolds-number regime, they provide numerical evidence for the propagation of Hasimoto-type solitons in a viscous fluid. In particular, the measured propagation velocity for large enough times is consistent with the LIA prediction
\[
v_{\mathrm{num}}
=\frac{\Gamma}{2\pi}\Lambda\tau_0,
\]
where $\Lambda$ is an effective LIA coefficient fitted from the numerically observed soliton trajectory. In our setting, recall that we have set $\lambda=1$, the analogous effective LIA factor is time-dependent and is initially given by
\[
\Lambda_{\mathrm{eff}}(t)
=\bigl|\log(\sqrt{\nu t})\bigr|,
\]
which arises from the desingularization of the Biot--Savart law at the viscous core scale $\sqrt{\nu t}$.

\end{document}